\documentclass[12pt]{amsart}
\pdfoutput=1

\usepackage{amsthm, amssymb, amsmath}
\usepackage[colorlinks=true, citecolor=blue, linkcolor=blue]{hyperref}
\usepackage{fullpage}
\usepackage{verbatim}
\usepackage{graphicx}
\usepackage{epsfig}
\usepackage{color}
\usepackage[all]{xy}
\newtheorem{thm}{Theorem}[section]
\newtheorem*{thmA}{Theorem~A}
\newtheorem*{thmB}{Theorem~B}
\newtheorem{lem}[thm]{Lemma}
\newtheorem{cor}[thm]{Corollary}
\newtheorem{prop}[thm]{Proposition}

\newtheorem*{conjecture*}{Conjecture}

\newtheorem*{guess*}{Philosophy}

\theoremstyle{remark} 
\newtheorem*{question*}{Question}

\newtheorem{remark}[thm]{Remark}

\theoremstyle{definition} 

\newtheorem*{convention*}{Convention}

\numberwithin{equation}{section}  % number equations by section

\newcommand\C{{\mathbb C}}
\newcommand\Chat { {\hat{\C}} } 
\newcommand\D{{\mathbb D}}

\renewcommand\P{{\mathbb P}}

\newcommand\eps{\varepsilon}
\renewcommand\gcd {\operatorname{gcd}} %gcd
\newcommand\Ratbar {\overline{\Rat}} 

\newcommand\del{\partial}

\newcommand{\OO}{\mathcal{O}}    % Fancy script O
\newcommand{\PP}{\mathbb{P}}      % Projective space
\newcommand{\XX}{\mathfrak{X}}      % Model of a variety - script X
\newcommand{\CC}{\mathbb{C}}      % Complex Numbers
\newcommand{\be}{\begin{equation}}
\newcommand{\ee}{\end{equation}}
\newcommand{\benn}{\begin{equation*}}
\newcommand{\eenn}{\end{equation*}}
\newcommand{\ba}{\begin{aligned}}
\newcommand{\ea}{\end{aligned}}
\newcommand{\bbm}{\begin{bmatrix}}
\newcommand{\ebm}{\end{bmatrix}}
\newcommand{\bpm}{\begin{pmatrix}}
\newcommand{\epm}{\end{pmatrix}}
\newcommand{\bi}{\begin{itemize}}
\newcommand{\ei}{\end{itemize}}
 
\newcommand{\ord}{\operatorname{ord}}
\newcommand{\supp}{\operatorname{supp}}   % Support
\newcommand{\an}[1]{\operatorname{an}}  % analytic space notation
 \newcommand{\Rat}{\mathrm{Rat}}    % space of rational maps
 \newcommand{\rhoH}{\rho_{\mathbf{H}}}    % metric on Berkovich hyperbolic space

\newcommand{\simarrow}{\stackrel{\sim}{\rightarrow}}    % Isomorphic mapping
\newcommand{\PGL}{\mathrm{PGL}}    % PGL
\newcommand{\Berk}{\mathbf{P}^1}  % Berkovich Projective Line
\newcommand{\red}{\mathrm{red}}
\newcommand{\DD}{\mathbb{D}}
\newcommand{\LL}{\mathbb{L}}

\renewcommand{\SS}{\mathcal{S}}

\newcommand{\EE}{\mathcal{E}}

\newcommand{\Ramtot}{\mathcal{R}^{\mathrm{tot}}}
\newcommand{\CCt}{\mathbb{C}(\!(t)\!)}
\renewcommand{\setminus}{\smallsetminus}

\title{Degenerations of Complex Dynamical Systems}

\author{Laura De Marco}
\address{Department of Mathematics, Statistics, and Computer Science \\
University of Illinois at Chicago \\
Chicago, IL}
\email{demarco@uic.edu}

\author{Xander Faber}
\address{Department of Mathematics\\
University of Hawaii at Manoa \\
Honolulu, HI}
\email{xander@math.hawaii.edu}

 

\begin{document}
\begin{abstract}
We show that the weak limit of the maximal measures for any degenerating sequence of rational maps on the Riemann sphere $\Chat$ must be a countable sum of atoms. For a 1-parameter family $f_t$ of rational maps, we refine this result by showing that the measures of maximal entropy have a \textit{unique} limit on $\hat\C$ as the family degenerates. The family $f_t$ may be viewed as a single rational function on the Berkovich projective line $\Berk_{\LL}$ over the completion of the field of formal Puiseux series in $t$, and the limiting measure on $\hat \C$ is the ``residual measure'' associated to the equilibrium measure on $\Berk_{\LL}$. For the proof, we introduce a new technique for quantizing measures on the Berkovich projective line and demonstrate the uniqueness of solutions to a quantized version of the pullback formula for the equilibrium measure on $\Berk_{\LL}$.  
\end{abstract}

\maketitle

%%%%%%%%%%%%%%%
%%%%%%%%%%%%%%%

\section{Introduction}

Let $f_k: \hat \C \to \hat \C$ be a sequence of endomorphisms of the Riemann sphere of degree $d\geq 2$ that diverges in the space of all endomorphisms.  Concretely, this means that at least one zero and pole of $f_k$ are colliding in the limit.  Our main goal is to understand the degeneration of the dynamical features of $f_k$ and, ultimately, to extract useful information from a ``limit dynamical system."  In this article, we concentrate on the measure of maximal entropy.  

 %to relate these degenerations to the properties of a limiting dynamical system.
 %construct a satisfactory notion of ``limit dynamical system.''
% dynamical properties of the sequence $f_k$ degeneration of $f_k$ as a sequence as dynamical systems and, ultimately, to find a meaningful limit.  %We wish to find a meaningful limit of this sequence as a degenerate dynamical system.

The existence and uniqueness of a measure of maximal entropy $\mu_f$ for a rational function $f$ of degree $\geq 2$ were shown in 1983 \cite{Lyubich_Measure, Freire-Lopes-Mane_Uniqueness_1983, Mane_Uniqueness_1983}.  Shortly after, Ma\~n\'e observed that the measure $\mu_f$ moves continuously in families \cite{Mane_Weakly_Continuous}, with the weak-$*$ topology of measures and the uniform topology on the space of rational functions.  By contrast, the Julia set $J(f) = \supp \mu_f$ fails to move continuously (in the Hausdorff topology) in the presence of bifurcations \cite{Mane:Sad:Sullivan}. 

	The space $\Rat_d$ of complex rational functions of degree $d \geq 2$ can be identified with the complement of a hypersurface in $\Ratbar_d = \PP^{2d+1}$. In \cite{DeMarco_Boundary_Maps_2005}, the first author showed that for ``most" degenerating sequences $f_k \to \del \Rat_d$, a limit of the maximal measures $\mu_{f_k}$ can be expressed as a countably-infinite sum of atoms.  (The measures $\mu_{f_k}$ themselves are atomless.)  There it was also shown that Ma\~n\'e's continuity property for maximal measures \textit{does not extend} to all of $\Ratbar_d$. Although weak limits of maximal measures for degenerating sequences may not be unique, our first main result shows that every weak limit is purely atomic.
	
	%defined an ``indeterminacy locus'' $I(d) \subset \del\Rat_d$ of dimension $d$ with the following property: if $f_k$ is a degenerating sequence of rational functions of degree~$d$ that converges to a point $f_0 \in \del\Rat_d \smallsetminus I(d)$, then the sequence of maximal measures $\mu_{f_k}$ converges weakly, and the limit measure is purely atomic (while $\mu_f$ is atomless for every $f \in \Rat_d$) . There it was also shown that Ma\~n\'e's continuity property for maximal measures \textit{does not extend} to the indeterminacy locus $I(d)$. Although weak limits of maximal measures for degenerating sequences may not be unique, our first main result shows that every weak limit is purely atomic.
% (This answers Question~6.2 of the first author in \cite{DeMarco_Boundary_Maps_2005}.)

\begin{thmA} \label{sequences}
Let $f_k$ be a sequence that diverges in the space $\Rat_d$ of complex rational functions of degree $d\geq 2$, and assume that the measures of maximal entropy $\mu_k$ converge to a probability measure $\mu$ on $\Chat$.  Then $\mu$ is equal to a countable sum of atoms.  
\end{thmA} 

Our second main result shows that Ma\~n\'e's continuity property \textit{does extend} to degenerating 1-parameter families.  Moreover, we are able to give a refined description of the limit measure using an associated dynamical system on the Berkovich projective line. 
	
\begin{thmB}  \label{main theorem}
Let $\{f_t: \; t\in \D\}$ be a meromorphic family of rational functions of degree $d \geq 2$ that is degenerate at $t=0$.  The measures of maximal entropy $\mu_t$ converge weakly on the Riemann sphere to a limiting probability measure $\mu_0$ as $t\to 0$.   The measure $\mu_0$ is equal to the residual equilibrium measure for the induced rational map $f: \Berk_{\LL} \to \Berk_\LL$ on the Berkovich projective line, where $\LL$ is the completion of the field of formal Puiseux series in~$t$.  
% In particular, the measure $\mu_0$ is a countable sum of atoms.
\end{thmB}

\begin{remark} 
The continuity of maximal measures on $\Chat$ can fail for degenerating families over a parameter space of dimension 2; see \cite[\S5]{DeMarco_Boundary_Maps_2005}.  
\end{remark}

\begin{remark}
	While we prefer to work with the more ``geometric'' field $\LL$, one can replace it with the field of formal Laurent series $\CCt$ in the statement of the theorem. 
\end{remark}
	
One should view the Berkovich dynamical system $(f, \Berk_{\LL})$ as the limit of dynamical systems $(f_t, \hat \C)$ as $t\to 0$.  This fruitful perspective was introduced by Jan Kiwi in his work on cubic polynomials and quadratic rational maps; see \cite{Kiwi_Puiseux_Dynamics_2006, Kiwi_Rescaling_Limits_2012} and \cite{Bonifant-Kiwi-Milnor}.  A closely related construction, viewing degenerations of polynomial maps as actions on trees, can be seen in \cite{DeMarco-McMullen_Trees}.  Charles Favre has recently constructed a compactification of the space of rational maps, where the boundary points are rational maps on a Berkovich $\Berk$ \cite{Favre:personalcomm}.  Our work is very much inspired by these results.  The Berkovich space viewpoint allows us to recover the results in \cite{DeMarco_Boundary_Maps_2005}, and it provides a conceptual explanation for the form of the limiting measures. In a sequel to this article, we will describe a countable--state Markov process that allows one to compute the residual measure explicitly.  

As with non-degenerating families, the Julia sets of $f_t$ may fail to converge to a limit as $t\to 0$.  Consider the example of $f_t(z) = t(z + z^{-1})$ in $\Rat_2$.  As $t\to 0$ along the real axis, the Julia set of $f_t$ is equal to the imaginary axis, while there is a sequence $t_n\to 0$ (tangent to the imaginary axis) for which $J(f_{t_n}) = \Chat$.  Ma\~n\'e used the continuity of $f\mapsto \mu_f$ to deduce that the Hausdorff dimension of $\mu_f$ is a continuous function of $f$, but this property does not extend to degenerating families; for example, the measures for a flexible Latt\`es family have dimension 2 while the limit measures always have dimension~0.  
	
The measure of maximal entropy $\mu_f$ for a rational function $f$ of degree $d\geq 2$ is characterized by the conditions that (a) it does not charge exceptional points, and (b) it satisfies the pullback relation 
	$$\frac{1}{d} \, f^*\mu_f = \mu_f.$$ 
To prove Theorem~A, we show that any weak limit of measures of maximal entropy on $\Chat$ must satisfy an appropriately-defined pullback formula (Theorem~\ref{paired pullback formula}); we then show that any measure satisfying this formula (for all iterates) is atomic.  The pullback formula is phrased in terms of ``paired measures,'' which is an ad hoc object we introduce to keep track of weak limits of measures in two sets of coordinates simultaneously. This is all accomplished in Section~\ref{Sec: Complex Surface}. 
	
The proof of Theorem~B (which inspired our proof of Theorem~A) is more conceptual and can be divided into three parts, each with its own collection of results that are of independent interest.  We sketch these results here.

\medskip\noindent
{\bf Step 1.  Dynamics on a complex surface.} 
 In Section~\ref{surface def}, we view the holomorphic family $f_t: \P^1 \to \P^1$ as one (meromorphic) dynamical system 
	$$F: X \dashrightarrow X$$
on the complex surface $X =  \D\times\P^1$, given by $(t,z) \mapsto (t, f_t(z))$ for $t\not=0$.  By hypothesis, $F$ will have points of indeterminacy in the central fiber $X_0 = \{0\}\times\P^1$.  If $F$ collapses $X_0$ to a point, we let $\pi: Y \to X$ be the (minimal) blow-up of the target surface so that $F: X \dashrightarrow Y$ is nonconstant at $t=0$; otherwise, set $Y =  X$ and $\pi = \mathrm{Id}$.   By counting multiplicities at the indeterminacy points of $F$, we define a notion of pullback $F^*$ from measures on the central fiber of $Y$ to measures on $X_0$.   We prove (Theorem~\ref{complex pullback}) that any weak limit $\nu$ of the measures $\mu_t$ on the central fiber of $Y$ satisfies a pullback relation: 
\begin{equation}  \label{eq 1}
	\frac{1}{d} F^* \nu = \pi_* \nu.
\end{equation}
%The proof is similar to the one given by Ma\~n\'e for continuity of the measures $\mu_f$, 
The proof relies on the Argument Principle to handle the measure at the points of indeterminacy for $F$. 
% Finally, there is an analog of exceptional points when $t=0$, and we show that weak limits $\nu$ do not charge these exceptional points.  

\medskip\noindent
{\bf Step 2.  Dynamics and $\Gamma$-measures on the Berkovich projective line.} 
Let $k$ be an algebraically closed field of characteristic zero that is complete with respect to a nontrivial non-Archimedean absolute value. The Berkovich analytification of the projective line $\PP^1_k$ will be denoted $\Berk$; it is a compact, Hausdorff, and uniquely arcwise connected topological space. A rational function $f: \PP^1_k \to \PP^1_k$ extends functorially to $\Berk$. If $d = \deg(f) \geq 2$, then the equilibrium measure $\mu_f$ may be characterized similarly to the complex case by the conditions that (a) it does not charge exceptional points of $\PP^1(k)$, and (b) it satisfies the pullback relation $\frac{1}{d} f^*\mu_f = \mu_f$ \cite{Favre_Rivera-Letelier_Ergodic_2010}. See \cite{Baker-Rumely_BerkBook_2010} for a reference specific to dynamics on $\Berk$, or see \cite{Berkovich_Spectral_Theory_1990} for the more general theory of non-Archimedean analytic spaces.  

	The goal of Section~\ref{Sec: Measures} is to define a notion of pullback $f^*$ on a new space of quantized measures relative to a finite set $\Gamma$ of vertices in $\Berk$. Every Borel probability measure $\nu$ on $\Berk$ gives rise to one of these ``$\Gamma$-measures'' $\nu_\Gamma$. And if $\nu$ is a solution to the standard pullback formula $\frac{1}{d}f^* \nu =  \nu$, then $\nu_\Gamma$ will satisfy a quantized version:
	\be
	\label{Eq: New Berkovich pullback}
			\frac{1}{d}f^* \nu_\Gamma =  \pi_* \nu_\Gamma.
	\ee
(One must push $\nu_\Gamma$ forward by a certain map $\pi$ in order to have a meaningful equation since $f^*\nu_\Gamma$ lies in the space of $\Gamma'$-measures for a potentially different vertex set $\Gamma'$.)  A solution to the pullback formula (\ref{Eq: New Berkovich pullback}) is typically far from unique. However, we will show (Theorem~\ref{Thm: One-vertex Unique}) that uniqueness is restored if one considers simultaneous solutions to pullback equations for all iterates of $f$, after ruling out measures supported on classical exceptional cycles.
%; the only possible solutions are the equilibrium measure for $f$ and measures supported on classical exceptional cycles. 

\medskip\noindent
{\bf Step 3.  A transfer principle.}
Now, let $k=\LL$ be the completion of the field of formal Puiseux series in~$t$, equipped with the non-archimedean absolute value that measures the order of vanishing at $t=0$. (See \cite[\S3]{Kiwi_Rescaling_Limits_2012}.) By viewing the parameter $t$ as an element of $\LL$, the family $f_t$ defines a single rational function $f$ with coefficients in $\LL$.  We define a vertex set $\Gamma\subset\Berk$ consisting of one vertex only, the Gauss point.  In \S\ref{residual measures}, we define a correspondence between measures on the central fiber of our surface $X$ with $\Gamma$-measures on $\Berk$.   From Step 1, any weak limit $\nu$ of the measures $\mu_t$ will satisfy the pullback relation \eqref{eq 1}.  The corresponding $\Gamma$-measure $\nu_\Gamma$ must satisfy the non-Archimedean pullback relation \eqref{Eq: New Berkovich pullback} on $\Berk$, by Proposition \ref{transfer}.   %A continuity argument in the complex setting shows that $\nu$ cannot charge exceptional points, 
We repeat the argument for all iterates $f_t^n$.  From Step 2, we deduce that $\nu_\Gamma$ is the equilibrium $\Gamma$-measure, and consequently, the limit measure $\nu$ is the ``residual" equilibrium measure.    See Section~\ref{Sec: Transfer}.
\medskip

\bigskip
\noindent
{\bf Acknowledgements.}  We are grateful for the opportunities that allowed these ideas to germinate: the 2010 Bellairs Workshop in Number Theory funded by the CRM in Montreal, and the Spring 2012 semester on Complex and Arithmetic Dynamics at ICERM. We would like to thank Charles Favre, Mattias Jonsson, Jan Kiwi, and Juan Rivera-Letelier for helpful discussions, and we further thank Jonsson for inviting us to speak about this work at the December 2012 RTG workshop at the University of Michigan. 

%%%%%%%%%%%%%%%
%%%%%%%%%%%%%%%
	
\bigskip
\section{The space of rational maps: complex-analytic arguments}
\label{Sec: Complex Surface}

In this section we prove Theorem~A along with a number of preliminary results that will be used in the first step of the proof of Theorem~B.

\subsection{The space of rational maps}  \label{space}
We will let $\Rat_d$ denote the set of all complex rational functions of degree $d$.  It can be viewed as an open subset of the complex projective space $\P^{2d+1}$, by identifying a function 
	$$f(z) = \frac{a_0z^d + a_1z^{d-1} + \cdots + a_d}{b_0z^d + b_1z^{d-1} + \cdots + b_d}$$ 
with its coefficients in homogeneous coordinates
	$$(a_0:a_1:\cdots:a_d:b_0:b_1:\cdots:b_d) \in \P^{2d+1}.$$
In fact, any point $\Phi \in \P^{2d+1}$ determines a pair $(P, Q)$ of homogeneous polynomials in two variables, and $\Rat_d = \P^{2d+1} \setminus \{\mathrm{Res}(P, Q) = 0\}$.  We set $\Ratbar_d = \P^{2d+1}$ so that $\del \Rat_d = \{\mathrm{Res}=0\}$.  For each $\Phi = (P,Q) \in \del\Rat_d$, we let $H = \gcd(P,Q)$, and let $\phi$ be the induced rational function of degree $< d$ defined by the ratio $P/Q$.   To match the algebraic language of the later sections, we refer to the map $\phi$ as the \textbf{reduction} of $\Phi$.  

A 1-parameter \textbf{holomorphic family} $\{f_t: \, t\in U\}$ is a holomorphic map from a domain $U\subset \C$ to $\Rat_d$.  A \textbf{meromorphic family} is a holomorphic map from $U$ to $\Ratbar_d$ with image not contained in $\del \Rat_d$.  A meromorphic family is \textbf{degenerate} at $u\in U$ if the image of $u$ lies in $\del \Rat_d$.

\begin{lem}  \label{nontrivial limit}
Let $f_k$ be a sequence in $\Rat_d$ converging to a point $\Phi \in \del \Rat_d$.  After passing to a subsequence if necessary, there is a sequence of M\"obius transformations $A_k \in \Rat_1$ so that $A_k \circ f_k$ converges in $\overline{\Rat}_d$ to a point with nonconstant reduction.  If $B_k$ is any other such sequence in $\Rat_1$, then $M_k = A_k \circ B_k^{-1}$ converges in $\Rat_1$ as $k\to\infty$ (along the subsequence determined by $A_k$).  If the $f_k$ lie in a meromorphic family $\{f_t: t \in \D\}$, then the sequence $A_k$ may be chosen to lie in a meromorphic family $\{A_t: t\in \D\}$. 
\end{lem}

\proof  
Existence is carried out, algorithmically, in \cite[Prop.~2.4]{Rivera-Letelier_Asterisque_2003} and appears also in \cite[Lemma 3.7]{Kiwi_Rescaling_Limits_2012} when the sequence lies in a holomorphic family; the strategy is as follows.  

At each step of this argument, we may pass to a  subsequence. Write 
	$$f_k(z,w) =   (P_k(z,w):  Q_k(z,w)),$$
normalized so that $(P_k, Q_k) \to (P,Q)$ in $\Ratbar_d$.  Note that at least one of $P$ and $Q$ is nonzero.  By replacing $f_k$ with $S_k \circ f_k$, where $S_k(z) = \alpha_k z$ with $\alpha_k>0$, it can be arranged that the limiting $P$ and $Q$ are both nonzero.  If $P$ is not a scalar multiple of $Q$, we are done.

Suppose $P = c_0 Q$ for some constant $c_0 \in \C^*$.  If $m = \deg_z P = \deg_z Q$, write 
	$$f_k(z,w) = (a_k z^mw^{d-m} + \hat{P}_k(z,w):  b_k z^mw^{d-m} + \hat{Q}_k(z,w))$$
where $\hat{P}_k$ and $\hat{Q}_k$ have no term involving $z^mw^{d-m}$.  Now, postcompose $f_k$ with a translation by $a_k/b_k = c_0 + o(1)$, replacing $f_k$ with
	$$f_k(z,w) = (P_k(z,w) - a_kb_k^{-1}Q_k(z,w):  Q_k(z,w)).$$
If $P$ and $Q$ are not monomials, then we are done; the new limit has nonconstant reduction.  If $P$ and $Q$ were monomials, the resulting limit in $\Ratbar_d$ will have constant reduction ($= 0$); we rescale and repeat the initial argument.  It follows that the new $P$ cannot be a scalar multiple of $Q$ because it has no term involving $z^mw^{d-m}$.  This completes the proof of existence of $\{A_k\}$.

If the given $f_k$  lies in a meromorphic family $f_t = (P_t, Q_t)$, then the scaling and translation maps can be chosen meromorphic in $t$, since they are built from the coefficients of $f_t$.  

Now suppose $A_k\circ f_k \to \Phi_A$ and $B_k \circ f_k \to \Phi_B$ in $\overline{\Rat}_d$, with nonconstant reductions $\phi_A$ and $\phi_B$. Set $M_k = A_k \circ B_k^{-1}$.  Again passing to a subsequence, $M_k$ converges to $M_0 \in \Ratbar_1$.  Away from finitely many points in $\PP^1$, we have
	\[
		\phi_A(p) = \lim_{k \to \infty} A_k \circ f_k(p) = \lim_{k\to\infty} M_k \circ B_k \circ f_k (p) = M_0 \circ  \phi_B(p).
	\]
As $\phi_A$ is nonconstant, so is $M_0$, and therefore $M_0 \in \Rat_1$.  This also shows that $M_0$ is uniquely determined, so the full sequence $M_k$ converges.  
\qed

\subsection{Counting pre-images}
\label{Sec: Counting pre-images}
Fix a sequence $f_k$ in $\Rat_d$, and assume that $f_k$ converges to a degenerate point $\Phi\in \del\Rat_d$ with gcd $H$ and {\em nonconstant} reduction $\phi$.  For each point $x\in \P^1$, we define multiplicities
\begin{equation}  \label{complex m and s}
	m(x) = \deg_x \phi \qquad \text{and} \qquad  s(x) = \ord_x H. 
\end{equation}
The quantity $m(x)$ is the local degree of $\phi$, and the quantity $s(x)$ will be called the {\bf surplus multiplicity} at $x$.

Let $\eta$ be a small loop around $\phi(x)$ bounding a disk $D$, and let $\gamma_x$ be the small loop around $x$ sent with degree $m(x)$ onto $\eta$ by $\phi$.   Choose $\gamma_x$ small enough so that it does not contain any roots of $H$, except possibly $x$ itself.  Because $f_k$ converges locally uniformly to $\phi$ on $\P^1\setminus \{H=0\}$, for each $k \gg 0$ there is a small loop $\gamma_k$ around $x$ that is mapped by $f_k$ with degree $m(x)$ onto $\eta$.  Let $U_k$ be the domain bounded by $\gamma_k$.

\begin{prop} \label{counting}
Assume that $f_k$ converges to $\Phi\in \del\Rat_d$ with nonconstant reduction.  Fix $x\in \PP^1$.  For all $k$ sufficiently large, 
	$$ \# ( f_k^{-1} (z_0) \cap \overline{U}_k) = m(x) + s(x)$$
and
	$$ \# ( f_k^{-1} (p_0) \cap \overline{U}_k) = s(x).$$
for all points $z_0$ in $\overline{D}$ and all points $p_0$ in $\P^1 \setminus \overline{D}$. 
\end{prop}

\proof
The proof is an application of the Argument Principle from complex analysis.  Assume first that $z_0=0\in D$ and $p_0 =\infty \not\in\overline{D}$.  Then
	$$\# ( f_k^{-1} (z_0) \cap U_k) =  \# ( f_k^{-1} (z_0) \cap \overline{U}_k)  = \#\; \mathrm{Zeroes}(f_k) \mbox{ inside } U_k,$$
and
	$$ \# ( f_k^{-1} (p_0) \cap U_k) =  \# ( f_k^{-1} (p_0) \cap \overline{U}_k) = \# \; \mathrm{Poles}(f_k) \mbox{ inside } U_k.$$
By the Argument Principle, for all large $k$ we have 
	$$\#( f_k^{-1} (z_0) \cap U_k)  - \#  ( f_k^{-1} (p_0) \cap U_k) \; =  \; \int_{\gamma_k} \frac{f_k'}{f_k}  \; =  \; m(x).$$ 
On the other hand, we may compute directly that
\begin{equation*} \label{poles}
	s(x) = \# \; \mathrm{Poles}(f_k) \mbox{ inside } U_k 
\end{equation*}
for all sufficiently large $k$, since $f_k \to \Phi$. Indeed, $H(x) = 0$ with multiplicity $s(x)$ (and $\phi(x) \not=\infty$), so there are exactly $s(x)$ poles converging to $x$ as $k \to \infty$.  (Compare \cite[Lem.~14]{DeMarco_Boundary_Maps_2005}.)   It remains to handle the case where $z_0\in \eta = \del D$.  By construction, the boundary $\gamma_k$ of $U_k$ is mapped with degree $m(x)$ over $\eta$; and by viewing $z_0$ as the point $\infty$, we see that there must be $s(x)$ preimages of $z_0$ converging to $x$ as $k\to \infty$.  
\qed

\subsection{Paired measures}
\label{Sec: Paired measures}
Let $C,E$ be two copies of $\PP^1$.  A \textbf{paired measure} $(\mu_C, \mu_E)$ is a pair of Borel probability measures $\mu_C$ on $C$ and $\mu_E$ on $E$.  Let $\{A_k\}$ be a sequence of M\"obius transformations in $\Rat_1$.   We say that a sequence of Borel probability measures $\{\mu_k\}$ on $\PP^1$ \textbf{converges $\{A_k\}$-weakly} to the paired measure $(\mu_C, \mu_E)$  if
	$$\mu_k \to \mu_C \qquad \mbox{ and } \qquad  A_{k*} \mu_k \to \mu_E$$
weakly. 

Let $\Phi$ be an element of $\Ratbar_d$ with reduction $\phi$.  We define a measure $\Phi^*(\mu_C,\mu_E)$ on $\PP^1$ by
the formula
	\benn %  \label{pullback paired}
		\Phi^*(\mu_C,\mu_E) := \phi^*\mu_E  + \sum_{x\in \PP^1} s(x)\delta_x.
	\eenn
Recall that $s(x)$ is defined in (\ref{complex m and s}).

\begin{lem}  \label{prob}
For any paired measure $(\mu_C,\mu_E)$, the measure $\Phi^*(\mu_C,\mu_E)$ has total mass $d$.  
\end{lem}

\begin{proof}
The proof is a simple degree count:
	\[
		\Phi^*(\mu_C,\mu_E)(\PP^1) = \deg(\phi) + \sum_{x \in \PP^1} s(x) = \deg(\phi) + \deg(H) = d.  \qedhere
	\]
\end{proof}

%Now suppose that the sequence $\{A_k\}$ diverges in $\Rat_1$, and assume that the limits $e_0 = \lim_{k\to\infty} A_k(z)$ and $c_0 = \lim_{k\to\infty} A_k^{-1}(z)$ exist (for all but finitely many points $z$).  If a sequence of measures $\mu_k$ converges $\{A_k\}$-weakly to the paired measure $(\mu_C, \mu_E)$, then necessarily we have 
%	$$\mu_C\{c_0\} + \mu_E\{e_0\} \geq 1.$$
%	
%\begin{remark}	The paired measures $(\mu_C, \mu_E)$ such that $\mu_C\{c_0\} + \mu_E\{e_0\} \geq 1$ are effectively the same as Borel probability measures on the quotient space $C \sqcup E / (c_0 \sim e_0)$ via push-forward by the quotient map.  The intuition behind these definitions arises from the geometric setting when the sequence $A_k$ lies in a 1-parameter meromorphic family, and $C$ and $E$ are irreducible components of the ``central fiber" of a surface $Y \to \D$. See \S\ref{Sec: Transfer}.
%\end{remark}

\subsection{Weak limits satisfy the pullback relation}
Fix a sequence $f_k$ in $\Rat_d$ that converges to $f_0 \in \del\Rat_d$.  We also fix the sequence $A_k$ of M\"obius transformations guaranteed by Lemma~\ref{nontrivial limit}, such that $A_k \circ f_k$ converges to a point $\Phi \in \Ratbar_d$ with gcd $H$ and reduction $\phi$ of degree $>0$.   If the reduction of $f_0$ is nonconstant,  we let $A_k$ be the identity for all $k$, so that $\Phi = f_0$.  (Note that if the reduction of $f_0$ is constant, it is possible that $\deg \phi = d$.) 

Let $C,E$ denote two copies of $\P^1$ as in \S\ref{Sec: Paired measures}.  If $f_0$ has nonconstant reduction, then $A_k(z) = z$ for all $k$ implies that $\mu_k \to (\mu_C, \mu_E)$ $\{A_k\}$-weakly if and only if $\mu_C = \mu_E$ and $\mu_k \to \mu_C$ weakly.

\begin{thm}  \label{paired pullback formula}
Any $\{A_k\}$-weak limit $(\mu_C, \mu_E)$ of the maximal measures $\mu_{f_k}$ will satisfy the pullback formula 
	$$\frac{1}{d} \, \Phi^*(\mu_C, \mu_E) = \mu_C$$
as measures on $C = \PP^1$. 
\end{thm}

\begin{proof}
	Without loss, we may replace $f_k$ with a subsequence in order to assume that $\mu_{f_k}$ converges $\{A_k\}$-weakly to $(\mu_C, \mu_E)$.  By the definition of $\{A_k\}$-weak convergence, and because $d^{-1} f_{k}^*\mu_k = \mu_k$ for all $k$, we know that 
\begin{equation} \label{convergence to pi*}
	d^{-1} f_{k}^*\mu_k \to \mu_C \quad \text{as} \quad k \to \infty.
\end{equation}
We need to show that the weak limit of $f_{k}^*\mu_k$ can also be expressed as $\Phi^*(\mu_C,\mu_E)$.  

Let $I(\Phi)$ denote the union of the roots of $H$. Let $U$ be a small neighborhood of $I(\Phi)$ in $\PP^1$.  Choose a partition of unity 
	$$b_r + b_s \equiv 1,$$
subordinate to the open cover $\{\PP^1\smallsetminus I(\Phi), U\}$
so that $b_r \equiv 1$ on $\P^1 \smallsetminus U$ and $b_s \equiv 1$ on a small neighborhood of $I(\Phi)$ inside $U$; as usual, $b_r$ and $b_s$ are non-negative continuous functions.  

Fix a non-negative continuous function $\psi$ on~$\PP^1$.  Recall that the pushforward of $\psi$ by $f\in\Rat_d$ can be defined by
	$$f_*\psi (y) = \sum_{f(x) = y} \psi(x),$$
where pre-images are counted with multiplicity.  Because $b_r$ vanishes near $I(\Phi)$, and because $A_k \circ f_k$ converges uniformly to $\phi$ on compact sets outside $I(\Phi)$, we have uniform convergence of functions
	$$(A_k \circ f_{k})_* (b_r\psi) \to \phi_* (b_r\psi),$$
and therefore
\begin{equation} \label{regular}
	\begin{aligned}
	\int b_r \psi \ (f_k^* \mu_k) &= \int b_r \psi \ \left( (A_k \circ f_k)^* A_{k*} \mu_k \right) \\
		&= \int \left(A_k \circ f_k \right)_* (b_r \psi) \ A_{k*} \mu_k \rightarrow \int \phi_*(b_r \psi) \ \mu_E 
		= \int b_r \psi \ \Phi^*(\mu_C, \mu_E), 
	\end{aligned}
\end{equation} 
by the weak convergence of $A_{k*} \mu_k$ to $\mu_E$.  Upon shrinking the neighborhood $U$, \eqref{convergence to pi*} and \eqref{regular} together will show that 
	\begin{equation} \label{notIndeterminate}
		\int_{\PP^1 \smallsetminus I(\Phi)} \psi \, \mu_C = \frac{1}{d} \int_{\PP^1 \smallsetminus I(\Phi)} \psi \, \Phi^*(\mu_C, \mu_E)
	\end{equation}
for any test function $\psi$.  

Fix $x\in I(\Phi)$.  As in \S\ref{Sec: Counting pre-images}, let $\eta$ be a small loop around $\phi(x)$ that bounds an open disk $D$, and let $\gamma_x$ be the small loop around $x$ sent with degree $m(x)$ onto $\eta$ by $\phi$.   Choose $\gamma_x$ small enough so that it does not contain any point in $I(\Phi)$ other than $x$ itself; we shall further assume that it is contained in the neighborhood where $b_s \equiv 1$.  Because $A_k\circ f_k$ converges locally uniformly to $\phi$ on $\P^1\smallsetminus I(\Phi)$, for each $k\gg0$ there is a small loop $\gamma_k$ around $x$ that is mapped by $f_k$ with degree $m(x)$ onto $\eta$; for large $k$, this $\gamma_k$ is also contained in the region where $b_s \equiv 1$.  Let $U_{x,k}$ be the domain bounded by $\gamma_k$.  

We now apply Proposition \ref{counting} to the sequence $A_k\circ f_k$.  For $x \in I(\Phi)$, let $\psi_{\inf}(x)$ denote the infimum of $\psi$ on the component of $U$ containing $x$. For all $k$ sufficiently large, 
\begin{align*}
	 \int_{\P^1} b_s\psi \ (f_{k}^*\mu_k) 
	 &\geq	 \sum_{x\in I(\Phi)} \psi_{\inf}(x) \int_{\overline{U}_{x,k}}  \ (A_k \circ f_k)^* A_{k*} \mu_k  \\
	&= 	  \sum_{x\in I(\Phi)} \psi_{\inf}(x) \int_{\P^1} \# \left( (A_k \circ f_k)^{-1}(y) \cap \overline{U}_{x,k} \right)  \, A_{k*}\mu_k(y) \\
	&=    \sum_{x\in I(\Phi)} \psi_{\inf}(x) \Big[  s(x)  A_{k*}\mu_k(\P^1\setminus \overline{D}) +  (m(x)+s(x))  A_{k*}\mu_k(\overline{D}) \Big]   \\
	 &=     \sum_{x\in I(\Phi)} \psi_{\inf}(x) \left[ s(x) + m(x) A_{k*}\mu_k(\overline{D})\right]. 
\end{align*}
Letting $k\to\infty$, the $\{A_k\}$-weak convergence of measures gives
	$$\liminf_{k\to\infty} A_{k*}\mu_k(\overline{D}) \geq  \mu_E(\{\phi(x)\}).$$
Because $d^{-1} f_{k}^*\mu_k$ converges weakly to $\mu_C$, we deduce that 
	$$\int b_s\psi \, \mu_C \geq \frac{1}{d} \sum_{x\in I(\Phi)} \left[ s(x) + m(x) \mu_E(\{\phi(x)\}) \right] \psi_{\min}(x).$$
Shrinking the neighborhood $U$ of $I(\Phi)$, we obtain
\begin{equation} \label{singular}
\int_{I(\Phi)} \psi \, \mu_C \geq \frac{1}{d} \sum_{x\in I(\Phi)}   \left[ s(x) 
		+  m(x) \, \mu_E(\{\phi(x)\}) \right] \psi(x)  = \frac{1}{d} \int_{I(\Phi)} \psi \, \Phi^*(\mu_C, \mu_E). 
\end{equation}
As $\psi$ was arbitrary, adding \eqref{notIndeterminate} to \eqref{singular} yields the inequality of positive measures 
	$$\mu_C \geq \frac{1}{d} \, \Phi^*(\mu_C, \mu_E).$$
  But both are probability measures (by Lemma~\ref{prob}), so we must have equality.
\end{proof}

\subsection{Proof of Theorem~A}
Let $f_k$ be a sequence in $\Rat_d$ converging to $f_0 \in \del \Rat_d$ and with maximal measures $\mu_k$ converging to a measure $\mu$.  From Lemma~\ref{nontrivial limit}, there is a sequence $A_k\in\Rat_1$ so that $A_k\circ f_k$ converges to $\Phi\in\Ratbar_d$ with reduction $\phi$ of positive degree. Passing to subsequences for each iterate $n$ and applying a diagonalization argument, we choose sequences $\{A_{n,k}: k\in\mathbb{N}\}$ in $\Rat_1$ so that
	$$A_{n,k} \circ f_k^n \to \Phi_n \quad \mbox { as } k\to\infty$$
in $\Ratbar_{d^n}$ with reduction $\phi_n$ so that $\deg \phi_n >0$ for every iterate $n$. %We extend this notation to $n = 0$ by setting $A_{0,k}(z) = z$ for all $k$.  
By sequential compactness of the space of probability measures on $\PP^1$ (and another diagonalization argument, if necessary), we may assume that $\mu_k$ converges $\{A_{n,k}\}$-weakly to a paired measure $(\mu, \mu_{E_n})$ for each $n \geq 1$. 

Since the measures $\mu_k$ are also the measures of maximal entropy for iterates $f_k^n$, Theorem~\ref{paired pullback formula} implies that 
	$$\mu(\{p\}) = \frac{1}{d^n}  \Phi^*_n(\mu, \mu_{E_n}) (\{p\}) \geq \frac{s_{\Phi_n}(p)}{d^n}$$
for any iterate $n$ and any point $p\in\P^1$; recall that the integers $s_{\Phi_n}(p)$ are defined in (\ref{complex m and s}).  Degree counting shows that $\sum_{p \in \PP^1} s_{\Phi_n}(p) = d^n - \deg \phi_n$, which yields
	$$
		1 \geq \sum_{p \in \PP^1} \mu(\{p\}) \geq 1 - \frac{\deg \phi_n}{d^n}.
	$$
If $ \deg\phi_n  = o(d^n)$ as $n\to \infty$, then we see immediately that $\mu$ is a countable sum of atoms.  It remains to treat the case where $\deg \phi_n \not= o(d^n)$.  

The next lemma shows that the reduction maps $\phi_n$ are not unrelated.

\medskip

\begin{lem}  \label{composition}
The reduction maps $\phi_n$ form a composition sequence.  That is, there exist rational functions $\phi_{n+1, n}$ of positive degrees $\leq d$ so that 
	$$\phi_{n+1}  = \phi_{n+1, n}\circ \phi_n$$
for each $n\geq 1$.  Moreover, $A_{n+1,k} \circ f_k \circ A_{n,k}^{-1}$ converges to $\phi_{n+1,n}$ away from finitely many points in $\P^1$.
\end{lem}

\begin{proof}
This lemma follows from uniqueness in Lemma \ref{nontrivial limit}.  Write $\Phi = H\phi$ for any $\Phi\in\Ratbar_d$, where $H$ is the gcd of the two polynomials defining $\Phi$ and $\phi$ is the reduction.  As $k\to\infty$, we have $A_{n,k} \circ f_k^n \to H_n\phi_n$ and $A_{n+1,k}\circ f_k^{n+1} \to H_{n+1} \phi_{n+1}$.  Consider the sequence $f_k\circ A_{n,k}^{-1}$ in $\Rat_d$.  Passing to a subsequence, there exists a sequence $C_k$ of M\"obius transformations so that $C_k\circ f_k\circ A_{n,k}^{-1} \to H\phi$ with $\deg \phi>0$.  But then, by the continuity of degenerate composition (exactly as in \cite[Lemma 2.6]{DeMarco_JAMS_2007}), we have
	$$(C_k\circ f_k \circ A_{n,k}^{-1} ) \circ (A_{n,k} \circ f_k^n) = C_k \circ f_k^{n+1} \to (H_n^d \cdot (H\circ \phi_n)) \; \phi\circ\phi_n.$$
But uniqueness in Lemma \ref{nontrivial limit} then implies that there exists a M\"obius transformation $B = \lim_{k\to\infty} A_{n+1, k}\circ C_k^{-1} $ so that $\phi_{n+1} = B\circ \phi\circ\phi_n$.  We set $\phi_{n+1,n} = B\circ \phi$.     
\end{proof}

\bigskip

Lemma \ref{composition} implies that the degree of $\phi_n$ may be computed by 
	\[
		\deg \phi_n = \deg \phi_1 \cdot \prod_{j = 1}^{n-1} \deg \phi_{j+1, j}
	\]
In particular, $\deg \phi_n \neq o(d^n)$ implies there exists $n_0>0$ so that $\deg \phi_{n+1, n} = d$ for all $n \geq n_0$. For the remainder of the proof, we will operate under this assumption.

Suppose for the moment that there exist nonnegative integers $m > n \geq n_0$ such that 
\be \label{preperiodic orbit}
	A_{n, k} \circ A_{m,k}^{-1} \to L \in \Rat_1 \quad \mbox{ as } k\to\infty
\ee
(after passing to a subsequence, if necessary).  From Lemma \ref{composition} and the continuity of composition,
	$$A_{n,k} \circ f_k^{m-n} \circ A_{n,k}^{-1} = A_{n,k}\circ A_{m,k}^{-1} \circ A_{m,k} \circ f_k^{m-n} \circ A_{n,k}^{-1} \longrightarrow 
  L \circ \phi_{m,m-1} \circ \cdots \circ \phi_{n+1,n},$$
and the limiting function has degree $d^{m-n}$.  In other words, the sequence of conjugates $A_{n,k} \circ f_k^{m-n} \circ A_{n,k}^{-1}$ will converge in $\Rat_{d^{m-n}}$.  But properness of the iteration map $\Rat_d \to \Rat_{d^{m-n}}$ \cite[Corollary 0.3]{DeMarco_Boundary_Maps_2005} implies that the sequence $A_{n,k} \circ f_k \circ A_{n,k}^{-1}$ must also converge uniformly to some rational function $g \in \Rat_d$.  The continuity of measures within $\Rat_d$ then implies that $\mu = \lim_{k\to\infty} (A_{n,k}^{-1})_* \mu_g$.  The sequence $\{A_{n,k}\}$ must diverge in $\Rat_1$ (because the sequence $\{f_k\}$ diverges in $\Rat_d$), so the limiting measure $\mu$ will be concentrated at a single point.

It remains to treat the case where 
		$$A_{m, k} \circ A_{n,k}^{-1}$$
diverges in $\Rat_1$ for all $m > n \geq n_0$. A diagonalization argument allows us to assume that the limit exists in $\Ratbar_1$, and we set 
	$$a_{m,n} := \lim_{k\to\infty}  A_{m,k} \circ A_{n,k}^{-1}(p)$$
for all but one point $p$ in $\P^1$, say $p = h_{m,n}$. Recall that we continue to assume that $\deg \phi_n = o(d^n)$ as $n\to\infty$, so there is a constant $0 < \kappa < 1$ such that $\deg \phi_n = \kappa d^n$ for all $n \geq n_0$. We wish to show that $\mu = \lim \mu_k$ is purely atomic. For the sake of a contradiction, we suppose otherwise and write 
$$\mu = \nu + \tilde \nu,$$
where $\tilde \nu$ is a countable sum of atoms and $\nu = \mu - \tilde \nu$ is a nonzero positive measure with no atoms. Similarly, write $\mu_{E_n} = \nu^n + \tilde \nu^n$, where $\nu^n$ and $\tilde \nu^n$ are the ``diffuse part'' and the ``atomic part'' of $\mu_{E_n}$, respectively. Applying Theorem~\ref{paired pullback formula} to the $n$th iterates $f_k^n$ and comparing diffuse parts, we find that 
	\[
		\nu = \frac{1}{d^n}\phi_n^*\nu^n \quad \Rightarrow \quad  0 < \nu(\PP^1) = \frac{\deg \phi_n}{d^n} \nu^n(\PP^1) = \kappa \cdot \nu^n(\PP^1)
	\]
for all $n \geq n_0$. Hence, there exists $N$ so that 
	$$\sum_{n = n_0}^N \nu^n(\P^1) \geq 2,$$

	Fix a small $\eps > 0$. For each pair $n_0 \leq m,n \leq N$ with $m \neq n$, choose small pairwise disjoint closed disks $D_{m,n}$ and $D'_{m,n}$ around $a_{m,n}$ and $h_{m,n}$, respectively. Let $U$ be the complement of all of these disks in $\PP^1$. Since $\nu^n$ is atomless, by shrinking $D_{m,n}$ and $D'_{m,n}$ as needed we may assume that
	\[
		\nu^n(U) > \nu^n(\PP^1) - \frac{\eps}{2^n} \qquad (n_0 \leq n \leq N).
	\]
Weak convergence of measures $(A_{n,k})_* \mu_k \to \mu_{E_n} = \nu^n + \tilde{\nu}^n$ implies that 
	$$(A_{n,k})_*\mu_k (U) > \nu^n(\P^1) -  \frac{\eps}{2^n}$$
for all sufficiently large $k$ and all $n_0 \leq n \leq N$.  (Restricting to finitely many $n$ allows us to do this uniformly.)

	For distinct indices $n_0 \leq m, n \leq N$, we have constructed $U$ to be disjoint from $D_{m,n}'$. It follows that $A_{m,k} \circ A_{n,k}^{-1}(U) \subset D_{m,n}$ for all $k \gg 0$, and hence $U \cap (A_{m,k} \circ A_{n,k}^{-1}(U)) = \varnothing$ for all sufficiently large $k$. Therefore, the sets
		\[
		A_{n_0,k}^{-1}(U), \, A_{n_0+1,k}^{-1}(U), \, \ldots \, , \, A_{N,k}^{-1}(U)	
		\]
are pairwise disjoint for all $k \gg 0$. (Again, restricting to finitely many sets allows us to do this uniformly.) But then 
	$$\mu_k (\P^1) \geq \sum_{n=n_0}^N \mu_k\left( A_{n,k}^{-1}(U)\right) 
		> \sum_{n=n_0}^N \left(\nu^n(\P^1) - \eps/2^n\right) > 2 - \eps > 1,$$
contradicting the fact that $\mu_k$ is a probability measure. This completes the proof of Theorem~A.

\begin{remark}
In the case where the sequence $f_k$ lies in a meromorphic family $f_t$, the condition that $\deg \phi_n \not= o(d^n)$ is characterized in the proof of Proposition \ref{Prop: surplus equidistribution}(2), in terms of dynamics on the Berkovich $\Berk$.
\end{remark}

%%%%%%%%%%%%%%%
%%%%%%%%%%%%%%%	

\bigskip
\section{1-parameter families and complex surfaces}
\label{surface def}

	In this section, we carry out Step 1 in the proof of Theorem~B. To start, we consider a meromorphic family $\{f_t \ : \ t \in \DD\}$ of rational functions of degree $d \geq 2$ and set up a geometric framework in which to talk about pullback of measures when $t=0$. Under the hypothesis of Theorem~B, the family $f_t$ defines a holomorphic disk in $\overline{\Rat}_d$ with $f_0\in \del\Rat_d$.  It is convenient to package the given 1-parameter family into one map on the complex surface $X = \D\times\P^1$, as
	$$F:  X \dashrightarrow X,$$
defined by $F(t,x) = (t,f_t(x))$ for $t\not=0$.  The map $F$ extends to a meromorphic map on the surface $X$ with a finite set of indeterminacy points in the central fiber $X_0 := \{0\}\times \P^1$.  The indeterminacy points coincide with roots of the polynomial $H_{f_0}$ defined in \S\ref{space}.  On any compact subset of $\P^1\setminus \{H_{f_0} = 0\}$, the functions $f_t$ converge uniformly to the reduction  $\phi_{f_0}$ as $t\to 0$. 

\subsection{The modified surface $Y$}
There is a unique (up to isomorphism) minimal modification $\pi: Y \to X$, so that the induced rational map 
	$$F: X \dashrightarrow Y$$
is nonconstant on $X_0$.  In other words, either $Y=X$, or we can blow up the image surface $X$ at a unique point of $X_0$ so that $F$ has no exceptional curve.  The resulting surface $Y$ may be singular.  In coordinates, the existence and uniqueness of $Y$ is immediate from Lemma~\ref{nontrivial limit}: If $(t,z)$ are local coordinates on $X$, then $(t,w)$ are local coordinates on $Y$, where $z = A_t(w)$. Moreover, we see that the central fiber $Y_0$ is reduced. 

\begin{figure} [ht]
\includegraphics[width=5in]{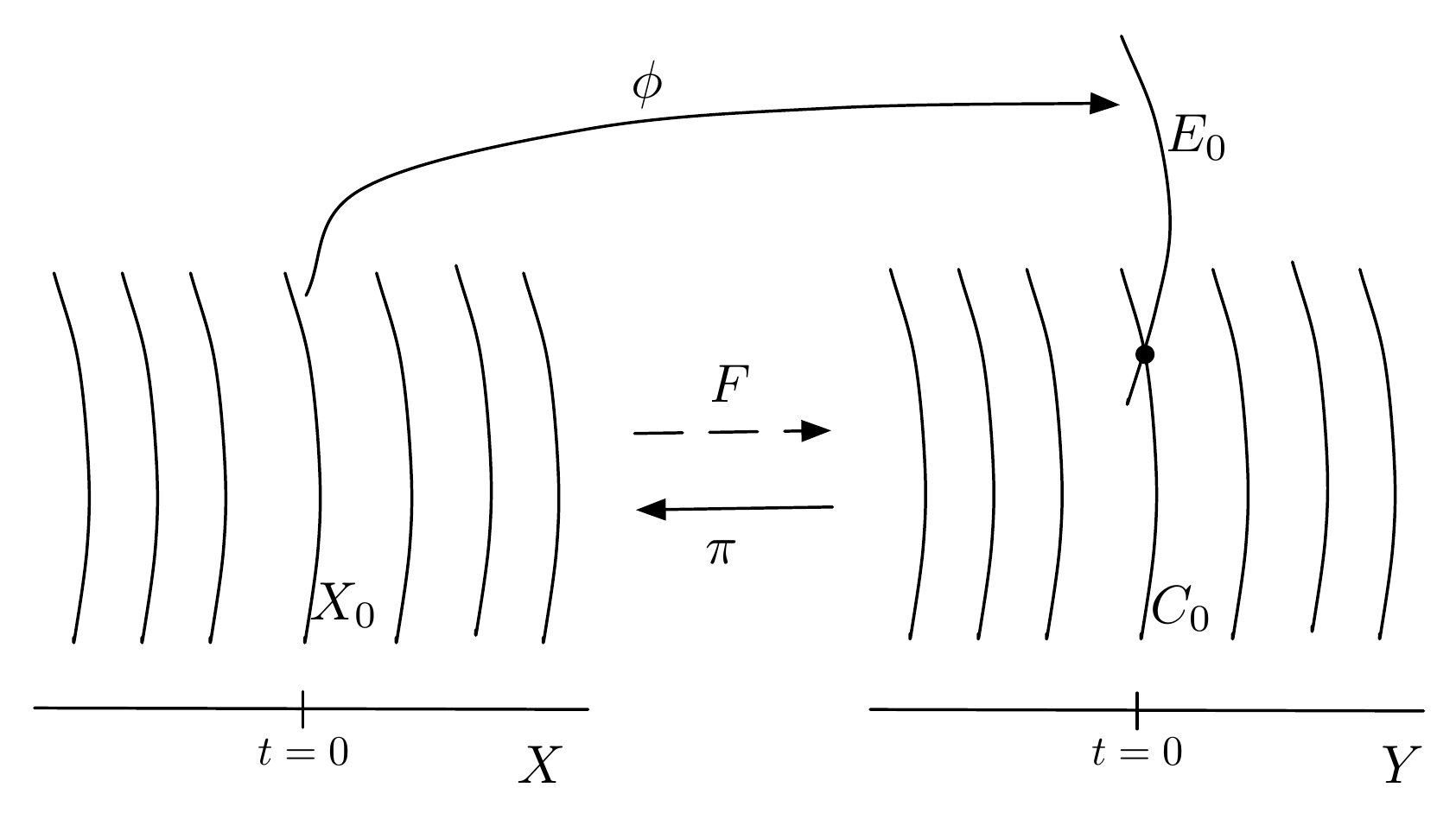}
\caption{The surface map $F: X\dashrightarrow Y$ when the given reduction $\phi_{f_0}$ is constant. }
\label{surface pic}
\end{figure}

We fix a family $A_t$ of M\"obius transformations --- as guaranteed by Lemma~\ref{nontrivial limit} --- such that $A_t \circ f_t$ converges to a point $\Phi \in \Ratbar_d$ with gcd $H$ and reduction $\phi$ of degree $>0$.   If the reduction of $f_0$ is nonconstant, we let $A_t$ be the identity for all $t$.  Away from its points of indeterminacy, the rational map $F:  X\dashrightarrow Y$ coincides with $\phi$ along the central fiber $X_0$.  The central fiber $Y_0$ of $Y$ has at most two irreducible components.  If $Y\not= X$, we let $E_0$ denote the exceptional curve of the projection $\pi$ and let $C_0$ be the other component of $Y_0$; see Figure~\ref{surface pic}.

\subsection{Pullback of measures from $Y_0$ to $X_0$}  \label{surface pullback}
For any Borel probability measure $\mu$ on the central fiber $Y_0$ of $Y$, we can define a measure $F^*\mu$ on the central fiber $X_0$ of $X$ of total mass~$d$.  We use the language of paired measures and their pullbacks defined in \S\ref{Sec: Paired measures}.  If $Y=X$, we simply set 
\be  \label{pullback X}
	F^*\mu := \Phi^* (\mu, \mu) = \, \phi^*\mu + \sum_{x\in\P^1} s(x) \delta_x
\ee
where $(\mu, \mu)$ is a paired measure on two copies of $Y_0 = X_0$.  In case $Y \not=X$, recall that the projection $\pi: Y\to X$ collapses $E_0$ to a point.  There is also a continuous projection $\pi_E: Y_0 \to E_0$ that collapses $C_0$ to a point.  We define, 
\be \label{pullback Y}
	F^*\mu := \Phi^* (\pi_* \mu, (\pi_E)_* \mu) =  \, \phi^*\, (\pi_E)_* \mu \, + \sum_{x\in\P^1} s(x) \delta_x .
\ee

Now suppose that $\mu_t$ is a family of probability measures on the fibers $Y_t$ on the surface $Y$.  We say $\mu_0$ on $Y_0$ is a weak limit of the measures $\mu_t$ if there is a sequence $t_n \to 0$ so that  
	$$\int_{Y_{t_n}} \psi \, \mu_{t_n} \to \int_{Y_0} \psi \, \mu_0$$
for every continuous function $\psi$ on $Y$.  If $Y=X = \D\times\P^1$, this notion of weak limit agrees with the usual notion for measures on a single $\P^1$. In case $Y\not=X$, it is not hard to see that this notion of convergence coincides with $\{A_{t_n}\}$-weak convergence of $\mu_{t_n}$ to the paired measure $(\pi_* \mu_0, (\pi_E)_*\mu_0)$ on $C_0 \cup E_0 = Y_0$.  %Recall that, by replacing $\psi$ with $\psi_+ - \psi_-$, it suffices to test convergence against non-negative functions.  

	We already know that weak limits of maximal measures satisfy a paired measure pullback formula (Theorem~\ref{paired pullback formula}). Translating into our surface framework, we immediately obtain the main result of this section:

\begin{thm}  \label{complex pullback}
Any weak limit $\mu_0$ of the maximal measures $\mu_t$ on the central fiber $Y_0$ of $Y$ will satisfy the pullback formula 
	$$\frac{1}{d} F^*\mu_0 = \pi_* \mu_0$$
on the central fiber $X_0$ of $X$.  
\end{thm}

\bigskip
\section{Dynamics and $\Gamma$-measures on the Berkovich projective line}
\label{Sec: Measures}

	In this section, we quantize a dynamical system $f$ on the Berkovich projective line and describe the solutions to a system of pullback formulas, thereby completing Step~2 of our program outlined in the introduction. Throughout, we let $k$ be an algebraically closed field of characteristic zero that is complete with respect to a nontrivial non-Archimedean absolute value. Only the case where $k$ has residue characteristic zero is necessary for our application; however, with essentially no extra work, we obtain a more general result. The Berkovich projective line over $k$ will be denoted $\Berk$ for brevity.

\subsection{Vertex sets and measures}
\label{Sec: Gamma}

	A \textbf{vertex set} for $\Berk$ is a finite nonempty set of type~II points, which we denote by $\Gamma$.  The connected components of $\Berk \smallsetminus \Gamma$ will be referred to as \textbf{$\Gamma$-domains}. When a $\Gamma$-domain has only one boundary point, we call it a \textbf{$\Gamma$-disk}. Write $\SS(\Gamma)$ for the partition of $\Berk$ consisting of the elements of $\Gamma$ and all of its $\Gamma$-domains. %The elements of $\SS(\Gamma)$ will be referred to as \textbf{$\Gamma$-ensembles}.  
	
	Let $(\Berk, \Gamma)$ be the measurable space structure on $\Berk$ equipped with the $\sigma$-algebra generated by $\SS(\Gamma)$. A measurable function on $(\Berk, \Gamma)$ will be called \textbf{$\Gamma$-measurable}. The space of complex measures on $(\Berk, \Gamma)$ will be denoted $M(\Gamma)$, and we call any such measure a \textbf{$\Gamma$-measure}. We write $M^\ell(\Gamma)$ for the convex subspace of $M(\Gamma)$ consisting of positive measures of volume~$\ell$. 
	
\begin{remark}
	A function $\phi: \Berk \to \CC$ is $\Gamma$-measurable if and only if it is constant on subsets of $\SS(\Gamma)$. 
\end{remark}

Suppose that $\Gamma \subset \Gamma'$ are two vertex sets.  If we write $\pi: \Berk \to \Berk$ for the identity morphism, then $\pi: (\Berk, \Gamma') \to (\Berk, \Gamma)$ is a measurable morphism. In particular, the projection
\begin{equation*}  \label{pi_*}
	\pi_* : M(\Gamma') \to M(\Gamma)
\end{equation*}
is $\CC$-linear and preserves positivity and volume of measures.

%%%%%%%%%%%%%%%
%%%%%%%%%%%%%%%

\subsection{Pulling back measures by a rational function}
\label{Sec: pullback}
Throughout this section we assume that $f: \Berk \to \Berk$ is a rational function of degree $d \geq 2$.  Suppose that $\Gamma = \{\zeta\}$ is a singleton vertex set, and let $\Gamma' = \{\zeta, f(\zeta)\}$ be a second vertex set.  For the applications in this article, we will only need to consider vertex sets of cardinality 1 or 2.

	Now we define a pullback map $f^*: M(\Gamma') \to M(\Gamma)$. As a first step, we define certain multiplicities $m_{U,V} \in \{0, 1, \ldots, d\}$ for each $U \in \SS(\Gamma')$ and $V \in \SS(\Gamma)$. If $V = \{\zeta\}$, set $m_{U,V} = m_f(\zeta)$, the usual local degree of $f$ at $\zeta$. For a $\Gamma$-disk $V$, we may write $V = D(\vec{v})$ for some tangent vector $\vec{v} \in T\Berk_{\zeta}$. Set $\bar f (V) = D(Tf(\vec{v}))$. Write $m_f(V)$ and $s_f(V)$ for the directional and surplus multiplicities for $f$ associated to $V$. (See \cite[\S3]{Faber_Berk_RamI_2013}.) By definition, we have 
	\[
		\#\left(f^{-1}(y) \cap V \right) = \begin{cases}
				m_f(V) + s_f(V) & \text{if } y \in \bar f (V) \\
				s_f(V) & \text{if } y \not\in \bar f(V).
			\end{cases}
	\]
Here we count each pre-image $x$ with multiplicity $m_f(x)$. Since $\bar f(V)$ is a union of elements of $\SS(\Gamma')$, the function $y \mapsto \#\left(f^{-1}(y) \cap V\right)$ is constant on elements of $\SS(\Gamma')$. For each $U \in \SS(\Gamma')$, define $m_{U,V}$ to be this constant value. The following lemma gives a compatibility relation among the multiplicities~$m_{U,V}$. 
	
\begin{lem}
\label{Lem: Consistent Multiplicities1}
	For each $U \in \SS(\Gamma')$, we have
		\[
			\sum_{V \in \SS(\Gamma)} m_{U,V} = \deg(f). 
		\]
\end{lem}

\begin{proof}
	Choose a point $y \in U$. For each $V \in \SS(\Gamma)$, we have that $m_{U,V} = \#\left(f^{-1}(y) \cap V\right)$. Since $f$ is everywhere $\deg(f)$--to--1, the result follows. 
\end{proof}		
	
	For a measurable function $\phi: (\Berk, \Gamma) \to \CC$, we define a $\Gamma'$-measurable function $f_*\phi$ by
	\begin{equation*}
	\label{Eq: pushforward2}
		f_*\phi(U) = \sum_{W \in \SS(\Gamma)} m_{U,W} \cdot \phi(W) \qquad (U \in \SS(\Gamma')).
	\end{equation*}
Here we have abused notation by writing $f_*\phi(U)$ for the constant value of $f_*\phi$ on $U$, and similarly for $\phi(W)$. Note that the sum defining $f_*\phi(U)$ is finite by Lemma~\ref{Lem: Consistent Multiplicities1}. 

	If $\phi$ is a bounded $\Gamma$-measurable function, then $\|f_*\phi\| \leq d \|\phi\|$, where we have written $\|\cdot\|$ for the sup norm.  For each $\nu \in M(\Gamma')$, the linear functional $\phi \mapsto \int f_*\phi \ \nu$ is bounded, and by duality there exists a $\Gamma$-measure $f^*\nu$ satisfying $\int \phi \ f^* \nu = \int f_*\phi \ \nu$ for all bounded $\Gamma$-measurable functions $\phi$. Evidently $f^*: M(\Gamma') \to M(\Gamma)$ preserves positivity of measures, and Lemma~\ref{Lem: Consistent Multiplicities1} shows that $f^*$ carries $M^\ell(\Gamma')$ into $M^{\ell d}(\Gamma)$ for each $\ell \in \CC$. In particular, $\frac{1}{d}f^*$ maps probability measures to probability measures.

%\begin{remark}
%\label{Rem: characteristic pushforward}
%	We will have use for the special case of~\eqref{Eq: pushforward2} in which $\phi = \chi_U$ is the characteristic function of some $U \in \SS(\Gamma)$:
%	\[
%		f_*\chi_U  = \sum_{V \in \SS(\Gamma')} m_{V,U} \cdot \chi_V. 
%	\]
%\end{remark}

%%%%%%%%%%%%%%%
%%%%%%%%%%%%%%%

\subsection{The equilibrium and exceptional $\Gamma$-measures}
\label{Sec: distinguished}
For a given rational function $f: \Berk \to \Berk$ of degree $d \geq 2$ and $\Gamma = \{\zeta\}$, there are two distinguished $\Gamma$-measures that will play a key role in our theory. 
	
	Write $\mu_f$ for the equilibrium measure on $\Berk$ relative to $f$ \cite{Favre_Rivera-Letelier_Ergodic_2010}. (Another common name in the literature is ``canonical measure'' \cite[\S10]{Baker-Rumely_BerkBook_2010}.) It is the unique Borel probability measure $\nu$ that satisfies $f^*\nu = d \cdot \nu$ and that does not charge classical points of $\Berk$ \cite[Thm.~A]{Favre_Rivera-Letelier_Ergodic_2010}. Here $f^*$ is the usual pullback operator for Borel measures on $\Berk$ --- not the one defined in \S\ref{Sec: pullback}. For a vertex set $\Gamma$, we define the \textbf{equilibrium $\Gamma$-measure} $\omega_{f,\Gamma}$ by the formula 
	$$\omega_{f, \Gamma}(U) := \mu_f(U)$$ 
for each $U \in \SS(\Gamma)$. Note that it is supported on a countable subset of $\SS(\Gamma)$. 

\begin{lem}
\label{Lem: canonical measure}
	Let $f: \Berk \to \Berk$ be a rational function of degree $d \geq 2$, let $\Gamma = \{\zeta\}$ be a singleton vertex set, let $\Gamma' = \{\zeta, f(\zeta)\}$, and let $\pi_*$ and $f^*$ be the operators defined in the previous section. Then 
	$
		\pi_* \omega_{f, \Gamma'} = \omega_{f, \Gamma}$ and $
		f^* \omega_{f, \Gamma'}  = d \cdot \pi_* \omega_{f, \Gamma'}.
	$  
\end{lem}

\begin{proof}
	The statement about $\pi_*$ is immediate from the definitions.
	
	Let $\phi: \Berk \to \CC$ be a $\Gamma$-measurable function. It is also Borel measurable on $\Berk$ since each element of $\SS(\Gamma)$ is either an open set or a point. The definitions of the multiplicities $m_{U,V}$ show that 
	\[
		f_*\phi(y) = \sum_{f(x) = y} m_f(x) \phi(x) \qquad (y \in \Berk),
	\]
which agrees with the formula for the pushforward of Borel measurable functions. 
Since $f^* \mu_f = d \cdot \mu_f$ as Borel measures on $\Berk$, we find that
	\[
		\int \phi \ f^* \omega_{f,\Gamma'} = \int f_* \phi \ \omega_{f,\Gamma'} = \int f_*\phi \ \mu_f = \int \phi \ f^*\mu_f = d \int \phi \ \mu_f = d \int \phi \ \pi_* \omega_{f,\Gamma'}.
	\]
Hence $f^*\omega_{f,\Gamma'} = d \cdot \pi_* \omega_{f,\Gamma'}$ as elements of $M(\Gamma)$. 
\end{proof}

Suppose now that the rational function $f: \Berk \to \Berk$ has an exceptional orbit $\EE$.  The \textbf{exceptional $\Gamma$-measure} associated to the orbit $\EE$ is defined to be the probability measure $\delta_\EE \in M(\Gamma)$ given by
	\[
		\delta_\EE(U) = \frac{ \#\left(\EE \cap U\right)}{\# \EE} .
	\]

\begin{remark}
Recall that an exceptional orbit $\EE$ is finite and $f^{-1}(\EE) = \EE$.  Since $k$ has characteristic zero, the function $f$ admits at most two classical exceptional points and at most one exceptional point in $\Berk \smallsetminus \PP^1(k)$ (necessarily of type~II). 
\end{remark}

\begin{lem}
\label{Lem: exceptional solutions}
	Let $f: \Berk \to \Berk$ be a rational function of degree $d \geq 2$, let $\Gamma = \{\zeta\}$ be a singleton vertex set, and let $\Gamma' = \{\zeta, f(\zeta)\}$. Suppose that $\mathcal{E}$ is an exceptional orbit for $f$. Write $\delta_\EE$ and $\delta_{\EE}'$ for the associated probability measures with respect to $\Gamma$ and $\Gamma'$, respectively. Then $\pi_* \delta_\EE' = \delta_\EE$ and $f^* \delta_{\EE}'  = d \cdot \pi_* \delta_{\EE}'$. 
\end{lem}

\begin{proof}
	Since exceptional measures count the number of exceptional points, we evidently have $\pi_* \delta_\EE' = \delta_\EE$. For the other equality, let $U \in \SS(\Gamma)$. Then
	\[
		f^*\delta_\EE'(U) 
		= \sum_{\substack{V \in \SS(\Gamma') \\ V \subset f(U)}} m_{V,U} 
			\frac{ \#(\EE \cap V)}{ \# \EE}.
	\]
The quantity $m_{V,U}$ is the constant value of $\#\left(f^{-1}(y) \cap U\right)$ for $y \in V$, counted with multiplicities. In particular, if $c \in \EE \cap V$, then  $m_{V,U} = 0$ or $d$, depending on whether $f^{-1}(c) \cap U$ is empty or not. Note also that $\#(\EE \cap U) = \#(\EE \cap f(U))$, since $\EE$ is a totally invariant set. Hence, 
	\[
		f^*\delta_\EE'(U) = \sum_{\substack{V \in \SS(\Gamma') \\ V \subset f(U)}}
			d \, \frac{\#(\EE \cap V)}{\#\EE}  = d \, \frac{ \#(\EE \cap U)
}{\#\EE}			= d \cdot \pi_* \delta_\EE'(U). \qedhere
	\]
\end{proof}

%%%%%%%%%%%%%%%
%%%%%%%%%%%%%%%

\subsection{Surplus equidistribution and surplus estimates}

	We now give two technical results that will be used to prove the main result in the next section. The first is of interest in its own right: it describes how surplus multiplicities of disks behave under iteration. The second gives a lower bound for the mass of a $\Gamma$-disk in terms of its surplus multiplicity. 
		
\begin{prop}[Surplus Equidistribution]
\label{Prop: surplus equidistribution}
	Let $f: \Berk \to \Berk$ be a rational function of degree $d \geq 2$ with associated equilibrium measure $\mu_f$. Suppose that the Julia set of $f$ is not equal to $\{\zeta\}$. Let $U$ be an open Berkovich disk with boundary point $\zeta$.  Then exactly one of the following is true:
	\begin{enumerate}
		\item
		\label{Item1} The iterated surplus multiplicities of $U$ satisfy
			\[
				s_{f^n}(U) = \mu_f(U) \cdot d^n + o(d^n).
			\]
		\item
		\label{Item2} The orbit $\OO_f(\zeta)$ converges along the locus of total ramification 
			to a classical exceptional orbit (of length~1 or~2), and 
				\[
					s_{f^n}(U) = 0 \text{ and } \mu_f\left(f^n(U) \right) = 1 \text{ for all $n\geq 1$.} 
				\]
	\end{enumerate}
\end{prop}

\begin{proof}	
	The two cases of the proposition are mutually exclusive. For if \eqref{Item2} holds, then $s_f(U) = 0$, so that $f(U) \neq \Berk$. The relation $\mu = \frac{1}{d} f^* \mu$ of Borel measures yields
	\[
		\mu(U) = \frac{m_f(U)}{d}\mu\left(f(U)\right) = \frac{m_f(U)}{d} > 0.
	\]
But then \eqref{Item1} is contradicted. 

	In the remainder of the proof, let us assume that case \eqref{Item1} of the proposition does not hold. The equilibrium measure $\mu_f$ does not charge~$\zeta$ by hypothesis on the Julia set of $f$. Let $y$ be an arbitrary point of $\Berk$ that is not a classical exceptional point. Using equidistribution of iterated pre-images \cite[Thm.~A]{Favre_Rivera-Letelier_Ergodic_2010}, we find that 
	\[
		\mu_f(U) = \lim_{n \to \infty} \frac{\#\left(f^{-n}(y) \cap U\right)}{d^n} = \lim_{n \to \infty} 
			\frac{\varepsilon(y, n, U) \cdot m_{f^n}(U) + s_{f^n}(U)}{d^n},
	\]
where $\varepsilon(y, n, U) = 1$ if $y \in \overline{f^n}(U)$ and $0$ otherwise.  We conclude that $m_{f^n}(U) \neq o(d^n)$; for otherwise, we are in case \eqref{Item1} of the proposition. 

	Let $\zeta_0 = \zeta$ and set $\zeta_n = f(\zeta_{n-1})$ for each $n \geq 1$. We can write $\vec{v}_n \in T\Berk_{\zeta_n}$ for the tangent vector such that $D(\vec{v}_0) = U$ and $Tf^n(\vec{v}_0) = \vec{v}_n$. Then 
	\[
		m_{f^n}(U) = \prod_{i = 0}^{n-1} m_f(D(\vec{v}_i)).
	\]
Each factor in the product is an integer in the range $1, \ldots, d$. If infinitely many of the multiplicities $m_f(D(\vec{v}_i))$ are strictly smaller than $d$, then $m_{f^n}(U) = o(d^n)$. Thus $m_f(D(\vec{v}_n)) = d$ for all $n \gg 0$. As multiplicities are upper semi-continuous, this shows $m_f(\zeta_n) = d$ for all $n$ sufficiently large, so that the orbit $\OO_f(\zeta)$ eventually lies in the locus of total ramification for~$f$. 

	We now show that $\OO_f(\zeta)$ converges to a classical exceptional orbit.  Let $n_0$ be such that $\zeta_n \in \Ramtot_f$ for all $n \geq n_0$. The locus of total ramification is connected \cite[Thm.~8.2]{Faber_Berk_RamI_2013}, and any pair of points in $\Ramtot_f$ lie at finite hyperbolic distance to each other unless one is a classical critical point. So it suffices to prove that the hyperbolic distance $\rhoH(\zeta_{n_0}, \zeta_n)$ grows without bound. For ease of notation, let us assume that $n_0 = 0$. Since $\zeta_n, \zeta_{n+1} \in \Ramtot_f$, the entire segment connecting them must lie in the locus of total ramification as well. Hence $f$ maps $[\zeta_n, \zeta_{n+1}]$ injectively onto $[\zeta_{n+1}, \zeta_{n+2}]$. Moreover, $\rhoH(\zeta_{n+1}, \zeta_{n+2}) = d \cdot \rhoH(\zeta_n, \zeta_{n+1})$. By induction, we see that 
		\[
			\rhoH(\zeta_{n + \ell}, \zeta_{n + \ell + 1}) = d^\ell \cdot \rhoH(\zeta_n, \zeta_{n+1}), \quad \ell = 0, 1, 2, \ldots,
		\]
so that the locus of total ramification has infinite diameter. The locus of total ramification has at most two classical points in it; hence, some classical totally ramified point $c$ is an accumulation point of $\OO_f(\zeta)$. By (weak) continuity of $f$, we find $f(c) \in \Ramtot_f$. So $c$ is exceptional of period~1 or~2. The orbit $\OO_f(\zeta)$ must actually converge to the orbit of $c$ since the latter is attractive. 

	Since $f$ has a classical exceptional point, it is conjugate either to a polynomial or to $z \mapsto z^{-d}$. We treat the former case and leave the latter to the reader. Without loss of generality, we now assume that $f$ is a polynomial and that $f^n(\zeta)$ converges to $\infty$ along the locus of total ramification. 
As $k$ has characteristic zero, the ramification locus near $\infty$ is contained in a strong tubular neighborhood of finite radius around $(\zeta_{0,R}, \infty)$ for some $R > 1$ \cite[Thm.~F]{Faber_Berk_RamII_2012}. Since hyperbolic distance is expanding on the ramification locus, we see that $f^n(\zeta)$ converges to infinity along the segment $(\zeta_{0,R}, \infty)$.
% else, $f^n(\zeta)$ will eventually be expelled from the locus of total ramification. 
In particular, since the Julia set of $f$ is bounded away from $\infty$, and since $f$ preserves the ordering of points in $\Berk$ relative to $\infty$, we see that $\zeta$ must lie above the entire Julia set. That is, every segment from a Julia point to $\infty$ must pass through $\zeta$.
	
	Note that if $V$ is a $\Gamma$-disk, then either $V$ does not meet infinity or it does not meet the Julia set (or both). In particular, the surplus multiplicities satisfy $s_{f^n}(U) = 0$ for all $n \geq 1$. Consequently, $U$ must meet the Julia set; else,  $\mu_f(U) = 0$ and we are in case~\eqref{Item1}. Observe that $\zeta \in f^n(U)$ for each $n \geq 1$, so that the entire Julia set of $f$ is contained in $f^n(U)$. This shows $\mu_f\left(f^n(U)\right) = 1$, and we are in case~\eqref{Item2} of the proposition as desired. 
\end{proof}

\begin{lem}[Surplus Estimate]
	Let $f: \Berk \to \Berk$ be a rational function of degree $d \geq 2$, and let $\Gamma = \{\zeta\}$ be a singleton vertex set. Set $\Gamma' = \{\zeta, f(\zeta)\}$. (Note that $\Gamma = \Gamma'$ is allowed.) For any $\Gamma$-disk $U$ and any $\Gamma'$-measure solution $\nu$ to the equation  $f^*\nu = d \cdot \pi_* \nu$, we find that 
	\[
		\nu(U) \geq \frac{s_f(U)}{d}.
	\]
\end{lem}

\begin{proof}
	For ease of notation, let us write $m = m_f(U)$ and $s = s_f(U)$. We may explicitly compute the multiplicities appearing in the pullback operator to be
 	\[
		m_{V,U} = 
				\begin{cases}
					m+s & \text{if $V \subset \bar f(U)$} \\
					s & \text{if $V \not\subset \bar f(U)$}. 
				\end{cases}
	\]	
Then for $\chi_U$ the characteristic function on the $\Gamma$-disk $U$, 
	\benn
		\ba
			d \cdot \pi_* \nu(U) = f^* \nu(U)  &= \int_{\bar f(U)} f_*\chi_U \ \nu 
				+ \int_{\Berk \smallsetminus \bar f(U)} f_*\chi_U \ \nu\\
				&= (m + s) \cdot \nu\left( \bar f(U) \right) 
					+ s \cdot \nu\left(\Berk \smallsetminus \bar f(U)\right) \\
				&= s + m \cdot \nu\left( \bar f(U) \right) \geq s. 
		\ea
	\eenn
Dividing by $d$ gives the result. 
\end{proof}

%%%%%%%%%%%%%%%
%%%%%%%%%%%%%%%

\subsection{Simultaneous solutions to iterated pullback formulas}
\label{Sec: Main Berkovich Theorem}

	The equation $f^* \nu = d \cdot \pi_* \nu$ does not necessarily have a unique solution $\nu \in M^1(\Gamma)$ as one might expect by analogy with the standard setting. However, the solution does become essentially unique if we impose all pullback relations $(f^n)^* \nu = d^n \cdot \pi_{n*} \nu$ for $n = 1, 2, 3, \ldots$
	
	Let $\Gamma = \{\zeta\}$ be a singleton vertex set for $\Berk$. Let $\Gamma_n = \{\zeta, f^n(\zeta)\}$ for each $n \geq 1$, and write $(f^n)^*$ and $\pi_{n*}$ for the pullback and pushforward operators relative to $\Gamma$ and $\Gamma_n$, respectively. We define a set of $\Gamma$-measures $\Delta_f \subset M^1(\Gamma)$  by
$$
		\Delta_f = \bigcap_{n \geq 1} \pi_{n*} \left\{ \omega \in M^1(\Gamma_n) : 
			(f^n)^*\omega = d^n \cdot \pi_{n*} \omega \right\}. 
$$
Each element of $\Delta_f$ is the projection of a solution to a pullback formula for each iterate of $f$, although we do not require any compatibility among these solutions. Linearity of the pullback and pushforward operators shows that $\Delta_f$ is a convex polyhedral set in the space $M^1(\Gamma)$.  Note that $\Delta_f$ is nonempty: since $\omega_{f,\Gamma} = \omega_{f^n, \Gamma}$, the set $\Delta_f$ must contain the equilibrium $\Gamma$-measure   $\omega_{f, \Gamma}$ (Lemma~\ref{Lem: canonical measure}). 

\begin{remark}
	The intersected sets that define $\Delta_f$ are typically not nested. 
\end{remark}

\begin{thm}
\label{Thm: One-vertex Unique}
	Let $f : \Berk \to \Berk$ be a rational function of degree $d \geq 2$, and let $\Gamma = \{\zeta\}$ be a singleton vertex set. Suppose that the Julia set of $f$ is not equal to $\{\zeta\}$. With the above notation, $\Delta_f$ is the convex hull of the equilibrium $\Gamma$-measure $\omega_{f,\Gamma}$ and at most one probability measure $\delta_{\EE}$ supported on a classical exceptional orbit~$\EE$. Moreover, if $\Delta_f \neq \{\omega_{f, \Gamma}\}$, then $f^n(\zeta)$ converges to an exceptional orbit along the locus of total ramification for $f$. 
\end{thm}

%\begin{remark}
%	If $\EE$ is an exceptional orbit for $f$, then the $\Gamma$-measure $\delta_\EE$ always lies in $\Delta_f$. The theorem implies that if $f^n(\zeta)$ does not converge to $\EE$ along the locus of total ramification, then $\delta_\EE = \omega_{f, \Gamma}$. 
%\end{remark}

\begin{remark}
	For our application to complex dynamics, it is sufficient to restrict to countably supported measures in the definition of $\Delta_f$. But the theorem shows that this hypothesis is unnecessary: an arbitrary $\Gamma$-measure satisfying all pullback formulas is countably supported. 
\end{remark}	

\begin{remark}
\label{Rem: Single Vertex p}
	With a little more work, one can show that this result continues to hold when $k$ has positive characteristic provided that $\OO_f(\zeta)$ does not converge to a wildly ramified exceptional orbit. 
\end{remark}

\begin{cor}
	With the hypotheses of Theorem \ref{Thm: One-vertex Unique}, no measure in $\Delta_f$ charges $\zeta$. 
\end{cor}

\begin{proof}
	The hypothesis on the Julia set guarantees that $\zeta$ is not exceptional and that $\mu_f$ does not charge $\zeta$. 
\end{proof}

\begin{proof}[Proof of Theorem~\ref{Thm: One-vertex Unique}]
	Suppose that $f^n(\zeta)$ does not converge along the locus of total ramification to a classical exceptional periodic orbit for $f$. Let $U$ be any $\Gamma$-domain for $\Gamma = \{\zeta\}$. If $\nu \in \Delta_f$, Proposition~\ref{Prop: surplus equidistribution} and the Surplus Estimate applied to $f^n$ and $U$ show that 
	\[
		\nu(U) \geq \frac{s_{f^n}(U)}{d^n} = \mu_f(U) + o(1). 
	\]
Since this is true for any $\Gamma$-disk $U$, and since $\mu_f$ is a probability measure with no support at $\zeta$, we conclude that $\nu(U) = \mu_f(U)$ for every $U \in \SS(\Gamma)$. 

	Now suppose that $f^n(\zeta)$ converges along the locus of total ramification to the orbit of a classical  exceptional point. Without loss, we may assume that the exceptional point is fixed by replacing $f$ with $f^2$. After conjugating the exceptional fixed point to $\infty$, we may assume that $f$ is a polynomial. As in the proof of Proposition~\ref{Prop: surplus equidistribution}, we find that $f^n(\zeta)$ converges to $\infty$ along the segment $(\zeta_{0,R}, \infty)$ for some $R > 1$, and $\zeta$ lies above the entire Julia set. 
	
	Suppose that $U$ is a $\Gamma$-domain that meets the Julia set. Then $f(U)$ contains the entire Julia set, and the standard pullback formula $f^* \mu_f = d \cdot \mu_f$ on $\Berk$ shows that
	\begin{equation}
	\label{Eq: mu on U}
		d \cdot \mu_f(U) = m_f(U) \cdot \mu_f(f(U)) = m_f(U) \quad \Rightarrow \quad \mu_f(U) = \frac{m_f(U)}{d}.
	\end{equation}
In particular, only finitely many $\Gamma$-disks may meet the Julia set. 
	
	Fix any $\nu \in \Delta_f$. Write $U_\infty$ for the unique $\Gamma$-domain containing infinity; write $U_1, \ldots, U_r$ for the $\Gamma$-domains that meet the Julia set; write $U_0$ for the union of the remaining elements of $\SS(\Gamma)$. Note that since we are in case~\eqref{Item2} of Proposition~\ref{Prop: surplus equidistribution}, the surplus multiplicity satisfies $s_{f^n}(U) = 0$ for all $n \geq 1$ and $U\in\SS(\Gamma)$. Furthermore, we observe that $f(U_\infty) \subset U_\infty$ and $m_{f^n}(U_\infty) = d^n$, and that $f^n$ maps $U_0$ onto $f^n(U_0) \subset U_\infty$ in everywhere $d^n$-to-1 fashion. 
	
	First we show that $\nu(U_0) = 0$. For each $n \geq 1$, there exists $\nu_n \in M^1(\Gamma_n)$ such that $(f^n)^* \nu_n = d^n \cdot \pi_{n*} \nu_n = d^n \cdot \nu$. Then
	\be
	\label{Eq: mass at infinity}
			d^n \cdot \pi_{n*} \nu_n(U_\infty) = (f^n)^* \nu_n(U_\infty)  
			= \int (f^n)_*\chi_{U_\infty} \ \nu_n 
				= d^n \cdot \nu_n\left( f^n(U_\infty) \right).
	\ee
Thus $\nu(U_\infty) = \nu_n\left( f^n(U_\infty) \right)$ for any $n \geq 1$. Write $A$ for the annulus with boundary points $\zeta$ and $f^n(\zeta)$. By definition of the pushforward, we see that
	\[
		\nu(U_\infty) = \pi_{n*} \nu_n(U_\infty) = \nu_n(f^n(U_\infty)) + 
			\nu_n(f^n(U_0)) + \nu_n(A). 
	\]
Therefore,  $\nu_n(A) = \nu_n(f^n(U_0)) = 0$. But the calculation \eqref{Eq: mass at infinity} applies equally well to $U_0$ to show that $\nu(U_0) = \nu_n\left( f^n(U_0) \right)$, and so we conclude that $\nu(U_0) = 0$.

	Next we observe that for $i = 1, \ldots, r$, we have
	\benn
			d^n \cdot \pi_{n*} \nu_n(U_i) = (f^n)^* \nu_n(U_i)  
				= \int (f^n)_*\chi_{U_i} \ \nu_n 
				= m_{f^n}(U_i)\cdot \nu_n\left( f^n(U_i) \right).
	\eenn
	From \eqref{Eq: mu on U}, we see that 
		\[
			\mu_f(U_i) = \frac{m_f(U_i)}{d} = \frac{m_{f^n}(U_i)}{d^n}, \quad n \geq 1, \quad i = 1, \ldots, r.
		\]		
Combining the last two displayed equations gives
	\[
		\nu(U_i) = \pi_{n*} \nu_n(U_i) = \mu_f(U_i) \nu_n\left( f^n(U_i) \right).
	\]
The quantity $a := \nu_n\left( f^n(U_i) \right)$ is independent of $n$ and $i$ since $\mu_f(U_i) > 0$ for $i = 1, \ldots, r$ and $f^n(U_1) = \cdots = f^n(U_r)$. Setting $b = \nu(U_\infty)$, we have proved that
		$\nu = a \cdot \omega_{f, \Gamma}  + b \cdot \delta_\infty$. 
% (Recall that $\delta_\infty \in \Delta_f$ by Lemma~\ref{Lem: exceptional solutions}). 		
\end{proof}

%%%%%%%%%%%%%%%
%%%%%%%%%%%%%%%

\bigskip
\section{A transfer principle}
\label{Sec: Transfer}

	In this section, we complete the proof of Theorem~B.  We explain the transfer of solutions of the pullback formula for dynamics on the our complex surfaces to $\Gamma$-measure solutions of the pullback formula on~$\Berk$, and vice versa.
	
\subsection{Reduction and the residual measures}  \label{residual measures}
	Let $X \to \DD$ be a proper fibered surface over a complex disk with generic fiber $\PP^1_\CC$. Assume that the fiber $X_0$ over the origin is reduced. Let $\LL$ be the completion of an algebraic closure of $\CCt$ endowed with the natural non-Archimedean absolute value, and write $\LL^\circ$ for its valuation ring. We claim that $X$ gives rise, canonically, to a vertex set $\Gamma \subset \Berk$. The local ring of $\DD$ at the origin is contained inside $\LL^\circ$, and hence so is its completion. By completing along the central fiber $X_0$ and base extending to $\LL^\circ$, we obtain a formal scheme $\XX$ over $\LL^\circ$ with generic fiber $\Berk = \Berk_{\LL}$. Note that since $X_0$ is reduced, it may be identified with the special fiber $\XX_s$ as $\CC$-schemes.  Let 
		$$\red_X: \Berk \to X_0$$ 
be the surjective reduction map \cite[2.4.4]{Berkovich_Spectral_Theory_1990}. Let $\eta_1, \ldots, \eta_r$ be the generic points of the irreducible components of the special fiber $X_0$. There exist unique type~II points $\zeta_1, \ldots, \zeta_r \in \Berk$ such that $\red_X(\zeta_i) = \eta_i$ for $i = 1, \ldots, r$. The desired vertex set is $\Gamma = \{\zeta_1, \ldots, \zeta_r\}$. 

For each closed point $x \in X_0$, the formal fiber $\red_X^{-1}(x)$ is a $\Gamma$-domain, as defined in \S\ref{Sec: Gamma}.   The association $x \mapsto  \red_X^{-1}(x)$ induces a bijection between points of the scheme $X_0$ and elements of $\SS(\Gamma)$.  We obtain a projection of measures, 
	$$\red_X^*:  M^1(X_0) \to M^1(\Gamma),$$
where $M^1(X_0)$ is the space of Borel probability measures on $X_0(\CC)$ (with its analytic topology) and $M^1(\Gamma)$ is the space of positive $\Gamma$-measures of total mass~1 on $\Berk$, defined as follows.  Given $\mu\in M^1(X_0)$, let $B = \{x \in X_0(\CC) : \mu(\{x\}) > 0\}$. The set $B$ is at most countable. Write $\eta_1, \ldots, \eta_r$ for the generic points of the irreducible components $C_1, \ldots, C_r$ of $X_0$. Define $\omega = \red_X^*(\mu)$ by
	\[
		\omega\left( \red_X^{-1}(x)\right) := \begin{cases}
				\mu\left( x  \right) & \text{if }  x \in X_0(\CC) \\
				\mu(C_i \smallsetminus B) & \text{if $x = \eta_i$ for some $i = 1, \ldots, r$}.
			\end{cases}
	\]
Evidently,  $\omega(\Berk) = \mu(X_0) =  1$. 

Now let $M^1(\Gamma)^{\dagger} \subset M^1(\Gamma)$ be the subset of $\Gamma$-measures that assign no mass to the elements of $\Gamma$.  The reduction map $\red_X$ induces
	$$\red_{X*}:  M^1(\Gamma)^\dagger \to M^1(X_0)$$
as a partial inverse to $\red_X^*$.  Explicitly,  the \textbf{residual measure} $\mu = \red_{X*}(\omega) \in M^1(X_0)$ is defined by 
	$$\mu(\{x\}) := \omega(\red_X^{-1}(x)) \qquad \left(x \in X_0(\CC) \right). $$
For each $\omega \in M^1(\Gamma)^\dagger$, the residual measure $\mu$ is an atomic probability measure on $X_0$.  
The terminology is explained by the case where $X_0$ is irreducible and $\Gamma = \{\zeta_{0,1}\}$ is the Gauss point of $\Berk$; the mass of the residual measure at a closed point $x \in X_0(\CC)$ is precisely the volume of the residue class $\red_X^{-1}(x) \subset \Berk$.

% \begin{remark}
%	It is not technically necessary to pass from the field $\CCt$ to the field $\LL$ in order to carry out the above analysis. % The reason for this choice is that all of the non-Archimedean literature we cite assumes the ground field is algebraically closed. 
% \end{remark}

%%%%%%%%%%%%%%%
%%%%%%%%%%%%%%%

\subsection{Compatibility of pullbacks}

	Let $f_t$ be a 1-parameter family of dynamical systems of degree $d \geq 2$ with $t$ varying holomorphically in a small punctured disk $\DD^*$ and extending meromorphically over the puncture. As in \S\ref{surface def}, we let $X = \DD \times \PP^1(\CC)$ and write $\pi: Y \to X$ for the minimal modification of $X$ along $X_0$ such that the induced rational map $F: X \dashrightarrow Y$ is not constant along $X_0$. The surfaces $X$ and $Y$ induce vertex sets $\Gamma = \{\zeta\}$ and $\Gamma' = \{\zeta, f(\zeta)\}$ on $\Berk = \Berk_{\LL}$, where $\LL$ is the completion of an algebraic closure of $\CCt$ endowed with the natural non-Archimedean absolute value, and the family $f_t$ defines $f: \Berk \to \Berk$. The pullback $F^*$ from measures on $Y_0$ to measures on $X_0$ is given by the formula (\ref{pullback X}) or (\ref{pullback Y}), depending on whether or not $f$ fixes $\zeta$.  %The goal of this section is to show that any measure $\mu$ on $Y_0$ satisfying the pullback formula for $F$ induces a $\Gamma'$-measure $\omega$ on $\Berk$ for an appropriate $\Gamma'$, and that $\omega$ satisfies the pullback formula on $\Berk$ relative to $\Gamma$ and $\Gamma'$. 

\begin{prop}[Transfer Principle]  \label{transfer}
	Let $F: X \dashrightarrow Y$, $f:\Berk\to\Berk$, $\Gamma$, and $\Gamma'$ be as above. The following conclusions hold. 
	\begin{enumerate}
		\item If $\mu$ is a measure on the central fiber $Y_0$ such that $F^*\mu = d \cdot \pi_* \mu$, then $\omega = \red_Y^*\mu$ is a $\Gamma'$-measure satisfying $f^*\omega = d \cdot \pi_* \omega$.
		\item If $\omega$ is a countably supported $\Gamma'$ probability measure satisfying $\omega(\Gamma') = 0$ and $f^*\omega = d \cdot \pi_* \omega$, then the residual measure $\mu = \red_{Y*} (\omega)$ satisfies $F^*\mu = d \cdot \pi_* \mu$. 
	\end{enumerate} 
\end{prop}

\begin{proof}
	We begin by comparing the notions of multiplicity defined for $F$ (on $X_0$) and for $f$ (on $\Berk$). Lemma~\ref{nontrivial limit} gives a meromorphic family of M\"obius transformations $A_t \in \PGL_2(\CC)$ for $t \in \DD$, holomorphic away from $t = 0$, such that $A_t \circ f_t$ converges as $t\to 0$ to $\Phi\in \Ratbar_d$ with nonconstant reduction $\phi$.  One one hand, this implies that $\phi$ describes the meromorphic map $F$ from the fiber $X_0$ onto its image component $E_0$ in $Y_0$ (or $C_0$ if $X = Y$). Evidently the local degree $m(x)$ for each point of $X_0$ may be read off algebraically as the order of vanishing of $\phi(z) - \phi(x)$ at $x$. On the other hand, we may view $A_t$ as an element $A \in \PGL_2\left(\CCt \right)$. In particular, $A \circ f$ has nonconstant reduction as a rational function on $\Berk$, and the reduction is equal to $\phi$. If $U_x$ is a $\Gamma$-disk with reduction $x \in X_0$, Rivera-Letelier's Algebraic Reduction Formula \cite[Cor.~9.25]{Baker-Rumely_BerkBook_2010} shows that the directional multiplicity $m_f(U_x)$ is equal to the order of vanishing of $\phi(z) - \phi(x)$ at $x$, and so we conclude that 
	\[
		m(x) = m_f(U_x). 
	\]
From the description of the surplus multiplicity of the map $\Phi$ in (\ref{complex m and s}) and the corresponding description of the surplus multiplicity in \cite[Lem.~3.17]{Faber_Berk_RamI_2013}, we also see that
	\[
		s(x) = s_f(U_x). 
	\]
Finally, the Algebraic Reduction Formula shows that $\deg(\phi) = m_f(\zeta)$, where $\zeta$ is the unique vertex in $\Gamma$.

	Since $\omega = \red^*_Y (\mu)$ is supported on countably many $\Gamma'$-domains, to prove the first statement of the Transfer Principle it suffices to show that 
	\be
	\label{Eq: To check1}
		F^* \mu = d \cdot \pi_* \mu \ \Rightarrow \ f^*\omega(U) = d \cdot \pi_* \omega(U) \text{ for every $U \in \SS(\Gamma)$}. \tag{TP1}
	\ee  
Under the hypotheses of the second statement of the Transfer Principle, we find that $\mu$ is countably supported. Thus it suffices to show that
	\be
	\label{Eq: To check2}
		f^* \omega = d \cdot \pi_* \omega \ \Rightarrow \ 
			F^*\mu(\{x\}) = d \cdot \pi_* \mu(\{x\}) \text{ for every closed point $x \in X_0$}. 
		\tag{TP2}
	\ee
We now prove these statements. 
	
\medskip	
	
\noindent	\textbf{Case $Y = X$.} 
Let $x \in X_0$ be a closed point, and write $U_x = \red_X^{-1}(x) \in \SS(\Gamma)$. Then 
	\benn
		\ba
			f^*\omega(U_x) &= \sum_{V \in S(\Gamma)} m_{V,U_x} \omega(V) \\
				&= m_{\bar f(U_x), U_x} \omega(\bar f (U_x)) 
					+ \sum_{V \in \SS(\Gamma)} s_f(U_x) \omega(V) \\
			&= m(x) \mu(\{\phi(x)\}) + s(x) \\
			&= F^* \mu(\{x\}), 
		\ea
	\eenn
while $d \cdot \omega(U_x) = d \cdot \mu(\{x\})$ by definition.  This immediately implies \eqref{Eq: To check2}.  

To verify \eqref{Eq: To check1}, it remains to consider the mass on $\Gamma = \Gamma'$.  Set $B = \{x \in X_0: \mu(x) > 0\}$.   If $F^*\mu = d\cdot \mu$, then $F^*\mu$ has no atoms in $X_0\smallsetminus B$.  From the definition of $F^*\mu$ in (\ref{pullback X}), we see that $F^*\mu$ agrees with $\phi^*\mu$ on $X_0\smallsetminus B$.  Thus, 
	\begin{equation*}
		\begin{aligned}
			d \cdot \omega(\zeta) = d\cdot \mu(X_0 \smallsetminus B) &= F^*\mu(X_0 \smallsetminus B) \\
			&= \phi^* \mu(X_0 \smallsetminus B) = \deg(\phi) \cdot \mu(X_0 \smallsetminus B) = m_f(\zeta) \omega(\zeta)
			= f^*\omega(\zeta).
		\end{aligned}
	\end{equation*}
% By assumption, we have $d > m_f(\zeta)$, so we conclude that $\omega(\zeta) = 0$.  
This proves \eqref{Eq: To check1} for all $U$ in $\SS(\Gamma)$.

\bigskip

\noindent	\textbf{Case $Y \neq X$.}	Recall that $Y_0 = C_0 \cup E_0$, where $C_0$ is the proper transform of $X_0$, and $E_0$ is the exceptional fiber of $\pi: Y\to X$.  In \S\ref{surface pullback}, to define $F^*\mu$ we introduced the (continuous) projection $\pi_E: Y_0 \to E_0$ that collapses $C_0$ to a point.

Let $U_x$ be the $\Gamma$-disk corresponding to a closed point $x \in X_0$. Recall that we set $\varepsilon(V,U_x) = 1$ or~$0$ depending on whether $\bar f(U_x) = V$ or not. We see that
	\benn
		\ba
			f^*\omega(U_x) &= \sum_{V \in \SS(\Gamma')} m_{V, U_x} \omega(V) \\
				&= \sum_{V \in \SS(\Gamma')} \left(s_f(U_x) + m_f(U_x)
					\varepsilon(V, U_x) \right)\omega(V) \\
				&= s(x) + m(x) \, (\pi_{E*}\mu) (\{\phi(x)\}) \\
				&= F^*\mu\left(\{x\}\right),
		\ea
	\eenn
while $\pi_* \omega(U_x) = \pi_* \mu\left(\{x\}\right)$ is immediate.  Evidently \eqref{Eq: To check2} follows, and \eqref{Eq: To check1} holds for all $\Gamma$-disks.  

To verify \eqref{Eq: To check1}, it remains to check the pullback relation for the mass on vertices.  Let $B = \{y \in Y_0 : \mu(\{y\}) > 0\}$, and set $B' = \pi(B \cup E_0) \subset X_0$.  Then 
$$	
			d \cdot \pi_* \mu \left(X_0 \smallsetminus B'\right) 
				= d \cdot \mu \left(\pi^{-1}(X_0 \smallsetminus B')\right) 
				= d \cdot \mu(C_0 \smallsetminus B). $$
If $F^*\mu = d\cdot \pi_* \mu$, there are no atoms of $F^*\mu$ outside $B'$.  From the definition of $F^*\mu$ in (\ref{pullback Y}), the measure $F^*\mu$ must agree with the pullback of $\pi_{E*}\mu$ by $\phi$ on the set $X_0 \smallsetminus B'$; therefore,
$$		F^*\mu(X_0 \smallsetminus B') = \phi^*(\pi_{E*}\mu) (X_0 \smallsetminus B')
			= \deg(\phi) \cdot \mu(E_0 \smallsetminus B).  $$
Putting these observations together yields 
	\benn
		f^*\omega(\zeta) = m_f(\zeta) \omega(f(\zeta)) 
			= \deg(\phi) \mu(E_0 \smallsetminus B)
			= d \cdot \mu(C_0\smallsetminus B) 
			= d \cdot \pi_*\omega (\zeta),
	\eenn
so that \eqref{Eq: To check1} is verified for all $U$ in $\SS(\Gamma)$.  
\end{proof}	
	
%%%%%%%%%%%%%%%
%%%%%%%%%%%%%%%	

\subsection{Proof of Theorem~B}
\label{Sec: Main Proof}
	
	We retain all of the notation from previous sections. 
	
	For each $n \geq 1$, let $F_n: X \dashrightarrow Y^n$ be the rational map of surfaces associated to the $1$-parameter family $f_t^n$ as constructed in \S\ref{Sec: Complex Surface}, and write $\pi_n: Y^n \to X$ for the blowing up morphism. Define
	\[
		\Delta_0 = \bigcap_{n \geq 1} \pi_{n*} \left\{ \mu \in M^1\left(Y^n_0\right) : 
			F_n^*\mu = d^n \cdot \pi_{n*} \mu \right\}  \subset M^1(X_0). 
	\]
Write $\omega_{f, \Gamma}$ for the equilibrium $\Gamma$-measure for $f: \Berk \to \Berk$ associated to $\Gamma = \{\zeta_{0,1}\}$.  Recall that $\Delta_f$ was defined in \S\ref{Sec: Main Berkovich Theorem}.  

\begin{thm}
\label{Thm: Complex solutions}
	Let $f_t$ be a meromorphic $1$-parameter family of rational functions of degree $d \geq 2$. Suppose that the family is not holomorphic at $t = 0$; i.e., $\deg(f_0) < d$. The reduction map induces a bijection 
	\[
		\red_X^*: \Delta_0 \simarrow \Delta_f,
	\]
with inverse given by the residual measure construction $\red_{X*}$.
\end{thm}

\begin{proof}
No measure in $\Delta_f$ charges the vertex $\zeta_{0,1} \in \Gamma$, and every measure in $\Delta_f$ is countably supported (Theorem~\ref{Thm: One-vertex Unique}). The Transfer Principle (applied to all iterates of $f_t$ and $f$) shows that the maps
	\[
		\red_X^* : \Delta_0 \to \Delta_f \qquad \text{and} \qquad \red_{X*} : \Delta_f \to \Delta_0
	\]
are well defined.  That they are inverse to one another follows from the definitions of $\red_X^*$ and $\red_{X*}$. 
\end{proof}

\begin{cor}
	With the setup of Theorem~\ref{Thm: Complex solutions}, $\Delta_0$ always contains the residual measure $\red_{X*}(\omega_{f,\Gamma})$, and $\Delta_0$ is either a point or a segment in the space of all probability measures. In the latter case, there exists a point mass $\delta_{p_0} \in \Delta_0$ and a $1$-parameter family of exceptional periodic points $p_t$ for $f_t$ such that $f_0$ is constant with value $p_0$, and $p_0$ is not an indeterminacy point for the rational map $F: X \dashrightarrow X$. 
\end{cor}

\begin{proof}
	Theorem~\ref{Thm: Complex solutions} allows us to transfer the statements about $\Delta_0$ to $\Delta_f$. The first statement is immediate from Theorem~\ref{Thm: One-vertex Unique}. If $\Delta_f \neq \{\omega_{f, \Gamma}\}$, then $f^n(\zeta)$ converges along the locus of total ramification to a classical exceptional orbit $\EE$. By replacing $f$ and $f_t$ with their second iterates if necessary, we may assume that $\EE = \{p\}$ is a single point. Now $p \in \PP^1(\CCt)$ by completeness. A priori, this gives a formal 1-parameter family $p_t$ with complex coefficients. Since $f_t(p_t) = p_t$ and $\frac{d f_t}{dz}(p_t) \equiv 0$, the implicit function theorem shows $p_t$ is a meromorphic $1$-parameter family in a small disk about $t = 0$. That is, $p = p_t$ is a 1-parameter family of exceptional fixed points for the family $f_t$. Since $f^n(\zeta)$ converges to $p$, and since $p$ is a super attracting fixed point for $f$, it follows that $f_0$ is constant with value equal to $p_0$. If $U$ is the open $\Gamma$-disk containing $p$, then $f(U) \subsetneq U$. In particular, this shows $s_f(U) = s(p_0) = 0$, so that $p_0$ is not an indeterminacy point for the rational map $F$. 
\end{proof}

	We are now ready to prove the second main result of the article. With the terminology we have set up in the preceding sections, our goal is to show that the family of measures of maximal entropy $\{ \mu_t : t \in \DD^*\}$ converges weakly to the residual measure $\red_{X*}(\omega_{f, \Gamma})$ as $t \to 0$, where $\Gamma = \{\zeta_{0,1}\}$ is the Gauss point of $\Berk$.  

\begin{proof}[Proof of Theorem~B]
	Let $\mu^1$ be any weak limit of the family $\mu_t$ of maximal measures as $t \to 0$ on the surface $Y^1$. 
% (which must exist by Prokhorov's Theorem). 
Fix a subsequence $(t_\ell)_{\ell \geq 1}$ so that $t_\ell \to 0$ and $\mu_{t_\ell} \to \mu^1$ weakly on $Y^1$. Set $\mu_0 = \pi_{1*} \mu^1$; then $\mu_{t_\ell} \to \mu_0$ weakly on $X$. 
	
	For each $n \geq 2$, let $\mu^n$ be a weak limit of the sequence $(\mu_{t_\ell})$ on the surface $Y^n$. Note that $\mu_0 = \pi_{n*} \mu^n$ by construction. Moreover, we have $F_n^* \mu^n = d^n \cdot \pi_{n*} \mu^n$ for all $n \geq 1$ (Theorem~\ref{complex pullback}). Hence $\mu_0 \in \Delta_0$. 
	
	It remains to prove that $\mu_0 = \red_{X*}(\omega_{f, \Gamma})$, the residual measure associated to $\omega_f$ and the vertex set $\Gamma$. This follows immediately from the preceding corollary unless there exists a family of exceptional periodic points $p_t$ for $f_t$, the reduction of $f_0$ is equal to the constant $p_0$, and $p_0$ is not indeterminate for the rational map $F$. In that case, $\mu_0 = a \cdot \red_{X*}( \omega_{f, \Gamma}) + b \cdot \red_{X*}( \delta_\EE)$, for some $a, b \geq 0$,  where $p_0 \in \supp (\red_{X*} (\delta_\EE))$. We must prove that $b = 0$. 
	
	Since $p_0$ is not indeterminate, by continuity there exists a neighborhood $N$ of $p_0$ such that $f_t(N) \subset N$ for all $t$ sufficiently close to zero. Hence, $N$ is contained in the Fatou set of $f_t$, and $\mu_t$ assigns no mass to $N$. By weak continuity, $\mu_0(N) = 0$. That is, $b = 0$ and $\mu_0 = \red_{X*}(\omega_{f, \Gamma})$ as desired. 
\end{proof}

%%%%%%%%

\bibliographystyle{plain}
\bibliography{xander_bib}

\providecommand\biburl[1]{\texttt{#1}}\def\cprime{$'$}
\begin{thebibliography}{10}

\bibitem{Baker-Rumely_BerkBook_2010}
Matthew Baker and Robert Rumely.
\newblock {\em Potential theory and dynamics on the {B}erkovich projective
  line}, volume 159 of {\em Mathematical Surveys and Monographs}.
\newblock American Mathematical Society, Providence, RI, 2010.

\bibitem{Berkovich_Spectral_Theory_1990}
Vladimir~G. Berkovich.
\newblock {\em Spectral theory and analytic geometry over non-{A}rchimedean
  fields}, volume~33 of {\em Mathematical Surveys and Monographs}.
\newblock American Mathematical Society, Providence, RI, 1990.

\bibitem{Bonifant-Kiwi-Milnor}
Araceli Bonifant, Jan Kiwi, and John Milnor.
\newblock Cubic polynomial maps with periodic critical orbit. {II}. {E}scape
  regions.
\newblock {\em Conform. Geom. Dyn.}, 14:68--112, 2010.

\bibitem{DeMarco_Boundary_Maps_2005}
Laura DeMarco.
\newblock Iteration at the boundary of the space of rational maps.
\newblock {\em Duke Math. J.}, 130(1):169--197, 2005.

\bibitem{DeMarco_JAMS_2007}
Laura DeMarco.
\newblock The moduli space of quadratic rational maps.
\newblock {\em J. Amer. Math. Soc.}, 20(2):321--355, 2007.

\bibitem{DeMarco-McMullen_Trees}
Laura~G. DeMarco and Curtis~T. McMullen.
\newblock Trees and the dynamics of polynomials.
\newblock {\em Ann. Sci. \'Ec. Norm. Sup\'er. (4)}, 41(3):337--382, 2008.

\bibitem{Faber_Berk_RamII_2012}
Xander Faber.
\newblock Topology and geometry of the {B}erkovich ramification locus for
  rational functions, {II}.
\newblock \textit{{M}ath. {A}nn.}, doi:10.1007/s00208-012-0872-3, 2012.

\bibitem{Faber_Berk_RamI_2013}
Xander Faber.
\newblock Topology and geometry of the {B}erkovich ramification locus for
  rational functions.
\newblock \textit{{M}anuscripta {M}ath.}, doi:10.1007/s00229-013-0611-4, 2013.

\bibitem{Favre:personalcomm}
Charles Favre.
\newblock Personal communication.

\bibitem{Favre_Rivera-Letelier_Ergodic_2010}
Charles Favre and Juan Rivera-Letelier.
\newblock Th\'eorie ergodique des fractions rationnelles sur un corps
  ultram\'etrique.
\newblock {\em Proc. Lond. Math. Soc. (3)}, 100(1):116--154, 2010.

\bibitem{Freire-Lopes-Mane_Uniqueness_1983}
Alexandre Freire, Artur Lopes, and Ricardo Ma{\~n}{\'e}.
\newblock An invariant measure for rational maps.
\newblock {\em Bol. Soc. Brasil. Mat.}, 14(1):45--62, 1983.

\bibitem{Kiwi_Puiseux_Dynamics_2006}
Jan Kiwi.
\newblock Puiseux series polynomial dynamics and iteration of complex cubic
  polynomials.
\newblock {\em Ann. Inst. Fourier (Grenoble)}, 56(5):1337--1404, 2006.

\bibitem{Kiwi_Rescaling_Limits_2012}
Jan Kiwi.
\newblock Rescaling limits of complex rational maps.
\newblock arXiv:1211.3397 [math.DS], preprint, 2012.

\bibitem{Lyubich_Measure}
M.~Ju. Ljubich.
\newblock Entropy properties of rational endomorphisms of the {R}iemann sphere.
\newblock {\em Ergodic Theory Dynam. Systems}, 3(3):351--385, 1983.

\bibitem{Mane:Sad:Sullivan}
R.~Ma{\~n}{\'e}, P.~Sad, and D.~Sullivan.
\newblock On the dynamics of rational maps.
\newblock {\em Ann. Sci. \'Ecole Norm. Sup. (4)}, 16(2):193--217, 1983.

\bibitem{Mane_Uniqueness_1983}
Ricardo Ma{\~n}{\'e}.
\newblock On the uniqueness of the maximizing measure for rational maps.
\newblock {\em Bol. Soc. Brasil. Mat.}, 14(1):27--43, 1983.

\bibitem{Mane_Weakly_Continuous}
Ricardo Ma{\~n}{\'e}.
\newblock The {H}ausdorff dimension of invariant probabilities of rational
  maps.
\newblock In {\em Dynamical systems, {V}alparaiso 1986}, volume 1331 of {\em
  Lecture Notes in Math.}, pages 86--117. Springer, Berlin, 1988.

\bibitem{Rivera-Letelier_Asterisque_2003}
Juan Rivera-Letelier.
\newblock Dynamique des fonctions rationnelles sur des corps locaux.
\newblock {\em Ast\'erisque}, (287):xv, 147--230, 2003.
\newblock Geometric methods in dynamics. II.

\end{thebibliography}

\end{document}